\documentclass[10pt,fleqn]{article}
\usepackage{makeidx}
\usepackage{cleveref}

\usepackage[dvips]{epsfig}
\input{epsf}
\usepackage{amsmath,amssymb,amsfonts}
\usepackage{amscd}
\usepackage[dvips,matrix,arrow]{xy}
\usepackage{array,tabularx}

\numberwithin{equation}{subsection}
\numberwithin{figure}{subsection}

\newtheorem{theorem}{Theorem}[subsection]

\newtheorem{remark}[theorem]{Remark}

\newtheorem{corollary}[theorem]{Corollary}

\def\span{\operatorname{Span}}%
\def\det{\operatorname{det}}%
\def\deg{\operatorname{deg}}%
\def\dim{\operatorname{dim}}%
\def\sp{\operatorname{Sp}}%
\def\min{\operatorname{min}}%
\def\gl{\operatorname{GL}}%
\def\sl{\operatorname{SL}}%
\def\inv{\operatorname{Inv}}%

\newenvironment{proof}[1][Proof:]{\begin{trivlist}
\item[\hskip \labelsep {\bfseries #1}]}{\end{trivlist}}

\newcommand{\qed}{\nobreak \ifvmode \relax \else
      \ifdim\lastskip<1.5em \hskip-\lastskip
      \hskip1.5em plus0em minus0.5em \fi \nobreak
      \vrule height0.75em width0.5em depth0.25em\fi}

 0
3  

\expandafter\ifx\csname ProvidesPackage\endcsname\relax\toks0=\expandafter{%
\fi\ProvidesPackage{diagrams}[2014/12/31 v3.94 Paul Taylor's
commutative diagrams]
\toks0=\bgroup}
\ifx\diagram\isundefined\else\expandafter\endinput \fi

\edef\cdrestoreat{
\noexpand\catcode`\noexpand\@=\the\catcode`\@
\noexpand\catcode`\noexpand\#=\the\catcode`\#
\noexpand\catcode`\noexpand\$=\the\catcode`\$
\noexpand\catcode`\noexpand\<=\the\catcode`\<
\noexpand\catcode`\noexpand\>=\the\catcode`\>
\noexpand\catcode`\noexpand\:=\the\catcode`\:
\noexpand\catcode`\noexpand\;=\the\catcode`\;
\noexpand\catcode`\noexpand\!=\the\catcode`\!
\noexpand\catcode`\noexpand\?=\the\catcode`\?
\noexpand\catcode`\noexpand\+=\the\catcode'53
}\catcode`\@=11 \catcode`\#=6 \catcode`\<=12 \catcode`\>=12
\catcode'53=12 \catcode`\:=12 \catcode`\;=12 \catcode`\!=12
\catcode`\?=12

\ifx\diagram@help@messages\CD@qK\let\diagram@help@messages y\fi

\def\cdps@Rokicki#1{\special{ps:#1}}\let\cdps@dvips\cdps@Rokicki\let
\cdps@RadicalEye\cdps@Rokicki\let\CD@HB\cdps@Rokicki\let\CD@IK\cdps@Rokicki
\let\CD@HB\cdps@Rokicki
\def\cdps@Bechtolsheim#1{\special{dvitps: Literal "#1"}}%
\let\cdps@dvitps\cdps@Bechtolsheim\let\cdps@IntegratedComputerSystems
\cdps@Bechtolsheim
\def\cdps@Clark#1{\special{dvitops: inline #1}}
\let\cdps@dvitops\cdps@Clark
\let\cdps@OzTeX\empty\let\cdps@oztex\empty\let\cdps@Trevorrow\empty
\def\cdps@Coombes#1{\special{ps-string #1}}

\count@=\year\multiply\count@12 \advance\count@\month
\ifnum\count@>24240 
\ifnum\count@>24243 
\errhelp{You may press RETURN and carry on for the time
being.}\errmessage{! remain compatible with your
files. Please get it NOW.}\fi\fi

\def\CD@DE{\global\let}\def\CD@RH{\outer\def}

{\escapechar\m@ne\xdef\CD@o{\string\{}\xdef\CD@yC{\string\}}
\catcode`\&=4 \CD@DE\CD@Q=&\xdef\CD@S{\string\&}
\catcode`\$=3 \CD@DE\CD@k=$\CD@DE\CD@ND=$
\xdef\CD@nC{\string\$}\gdef\CD@LG{$$}
\catcode`\_=8 \CD@DE\CD@lJ=_
\obeylines\catcode`\^=7 \CD@DE\@super=^
\ifnum\newlinechar=10 \gdef\CD@uG{^^J}
\else\ifnum\newlinechar=13 \gdef\CD@uG{^^M}
\else\ifnum\newlinechar=-1 \gdef\CD@uG{^^J}
\else\CD@DE\CD@uG\space\expandafter\message{! input error: \noexpand
\newlinechar\space is ASCII \the\newlinechar, not LF=10 or CR=13.}
\fi\fi\fi}

\mathchardef\lessthan='30474 \mathchardef\greaterthan='30476

\ifx\tenln\CD@qK
\font\tenln=line10\relax
\fi\ifx\tenlnw\CD@qK\ifx\tenln\nullfont\let\tenlnw\nullfont\else
\font\tenlnw=linew10\relax
\fi\fi

\ifx\inputlineno\CD@qK\csname
newcount\endcsname\inputlineno\inputlineno\m@ne
\fi

\def\cd@shouldnt#1{\CD@KB{* THIS (#1) SHOULD NEVER HAPPEN! *}}

\def\get@round@pair#1(#2,#3){#1{#2}{#3}}
\def\get@square@arg#1[#2]{#1{#2}}
\def\CD@AE#1{\CD@PK\let\CD@DH\CD@@E\CD@@E#1,],}
\def\CD@m{[}\def\CD@RD{]}\def\commdiag#1{{\let\enddiagram\relax\diagram[]#1%
\enddiagram}}

\def\CD@BF{{\ifx\CD@EH[\aftergroup\get@square@arg\aftergroup\CD@YH\else
\aftergroup\CD@JH\fi}}
\def\CD@CF#1#2{\def\CD@YH{#1}\def\CD@JH{#2}\futurelet\CD@EH\CD@BF}

\def\CD@KK{|}

\def\CD@PB{
\tokcase\CD@DD:\CD@y\break@args;\catcase\@super:\upper@label;\catcase\CD@lJ:%
\lower@label;\tokcase{~}:\middle@label;
\tokcase<:\CD@iF;
\tokcase>:\CD@iI;
\tokcase(:\CD@BC;
\tokcase[:\optional@;
\tokcase.:\CD@JJ;
\catcase\space:\eat@space;\catcase\bgroup:\positional@;\default:\CD@@A
\break@args;\endswitch}

\def\switch@arg{
\catcase\@super:\upper@label;\catcase\CD@lJ:\lower@label;\tokcase[:\optional@
;
\tokcase.:\CD@JJ;
\catcase\space:\eat@space;\catcase\bgroup:\positional@;\tokcase{~}:%
\middle@label;
\default:\CD@y\break@args;\endswitch}

\let\CD@tJ\relax\ifx\protect\CD@qK\let\protect\relax\fi\ifx\AtEndDocument
\CD@qK\def\CD@PG{\CD@gB}\def\CD@GF#1#2{}\else\def\CD@PG#1{\edef\CD@CH{#1}%
\expandafter\CD@oC\CD@CH\CD@OD}\def\CD@oC#1\CD@OD{\AtEndDocument{\typeout{%
\CD@tA: #1}}}\def\CD@GF#1#2{\gdef#1{#2}\AtEndDocument{#1}}\fi\def\CD@ZA#1#2{%
\def#1{\CD@PG{#2\CD@mD\CD@W}\CD@DE#1\relax}}\def\CD@uF#1\repeat{\def\CD@p{#1}%
\CD@OF}\def\CD@OF{\CD@p\relax\expandafter\CD@OF\fi}\def\CD@sF#1\repeat{\def
\CD@q{#1}\CD@PF}\def\CD@PF{\CD@q\relax\expandafter\CD@PF\fi}\def\CD@tF#1%
\repeat{\def\CD@r{#1}\CD@QF}\def\CD@QF{\CD@r\relax\expandafter\CD@QF\fi}\def
\CD@tG#1#2#3{\def#2{\let#1\iftrue}\def#3{\let#1\iffalse}#3}\if y%
\diagram@help@messages\def\CD@rG#1#2{\csname newtoks\endcsname#1#1=%
\expandafter{\csname#2\endcsname}}\else\csname
newtoks\endcsname\no@cd@help \no@cd@help{See the
manual}\def\CD@rG#1#2{\let#1\no@cd@help}\fi\chardef\CD@lF =1
\chardef\CD@lI=2 \chardef\CD@MH=5 \chardef\CD@tH=6 \chardef\CD@sH=7
\chardef\CD@PC=9 \dimendef\CD@hI=2 \dimendef\CD@hF=3
\dimendef\CD@mF=4 \dimendef\CD@mI=5 \dimendef\CD@wJ=6
\dimendef\CD@tI=8 \dimendef\CD@sI=9 \skipdef\CD@uB=1
\skipdef\CD@NF=2 \skipdef\CD@tB=3 \skipdef\CD@ZE=4 \skipdef \CD@JK=5
\skipdef\CD@kI=6 \skipdef\CD@kF=7 \skipdef\CD@qI=8 \skipdef\CD@pI=9
\countdef\CD@JC=9 \countdef\CD@gD=8 \countdef\CD@A=7 \def\sdef#1#2{\def#1{#2}%
}\def\CD@L#1{\expandafter\aftergroup\csname#1\endcsname}\def\CD@RC#1{%
\expandafter\def\csname#1\endcsname}\def\CD@sD#1{\expandafter\gdef\csname#1%
\endcsname}\def\CD@vC#1{\expandafter\edef\csname#1\endcsname}\def\CD@nF#1#2{%
\expandafter\let\csname#1\expandafter\endcsname\csname#2\endcsname}\def\CD@EE
#1#2{\expandafter\CD@DE\csname#1\expandafter\endcsname\csname#2\endcsname}%
\def\CD@AK#1{\csname#1\endcsname}\def\CD@XJ#1{\expandafter\show\csname#1%
\endcsname}\def\CD@ZJ#1{\expandafter\showthe\csname#1\endcsname}\def\CD@WJ#1{%
\expandafter\showbox\csname#1\endcsname}\def\CD@tA{Commutative
Diagram}\edef
\CD@kH{\string\par}\edef\CD@dC{\string\diagram}\edef\CD@HD{\string\enddiagram
}\edef\CD@EC{\string\\}\def\CD@eF{LaTeX}\ifx\@ignoretrue\CD@qK\expandafter
\CD@tG\csname if@ignore\endcsname\ignore@true\ignore@false\def\@ignoretrue{%
\global\ignore@true}\def\@ignorefalse{\global\ignore@false}\fi

\def\CD@g{{\ifnum0=`}\fi}\def\CD@wC{\ifnum0=`{\fi}}\def\catcase#1:{\ifcat
\noexpand\CD@EH#1\CD@tJ\expandafter\CD@kC\else\expandafter\CD@dJ\fi}\def
\tokcase#1:{\ifx\CD@EH#1\CD@tJ\expandafter\CD@kC\else\expandafter\CD@dJ\fi}%
\def\CD@kC#1;#2\endswitch{#1}\def\CD@dJ#1;{}\let\endswitch\relax\def\default:%
#1;#2\endswitch{#1}\ifx\at@\CD@qK\def\at@{@}\fi\edef\CD@P{\CD@o pt\CD@yC}%
\CD@RC{\CD@P>}#1>#2>{\CD@z\rTo\sp{#1}\sb{#2}\CD@z}\CD@RC{\CD@P<}#1<#2<{\CD@z
\lTo\sp{#1}\sb{#2}\CD@z}\CD@RC{\CD@P)}#1)#2){\CD@z\rTo\sp{#1}\sb{#2}\CD@z}%
\CD@RC{\CD@P(}#1(#2({\CD@z\lTo\sp{#1}\sb{#2}\CD@z}
\def\CD@O{\def\endCD{\enddiagram}\CD@RC{\CD@P A}##1A##2A{\uTo<{##1}>{##2}%
\CD@z\CD@z}\CD@RC{\CD@P V}##1V##2V{\dTo<{##1}>{##2}\CD@z\CD@z}\CD@RC{\CD@P=}{%
\CD@z\hEq\CD@z}\CD@RC{\CD@P\CD@KK}{\vEq\CD@z\CD@z}\CD@RC{\CD@P\string\vert}{%
\vEq\CD@z\CD@z}\CD@RC{\CD@P.}{\CD@z\CD@z}\let\CD@z\CD@Q}\def\CD@IE{\let\tmp
\CD@JE\ifcat
A\noexpand\CD@CH\else\ifcat=\noexpand\CD@CH\else\ifcat\relax
\noexpand\CD@CH\else\let\tmp\at@\fi\fi\fi\tmp}\def\CD@JE#1{\CD@nF{tmp}{\CD@P
\string#1}\ifx\tmp\relax\def\tmp{\at@#1}\fi\tmp}\def\CD@z{}\begingroup
\aftergroup\def\aftergroup\CD@T\aftergroup{\aftergroup\def\catcode`\@\active
\aftergroup
@\endgroup{\futurelet\CD@CH\CD@IE}}\newcount\CD@uA\newcount\CD@vA
\newcount\CD@wA\newcount\CD@xA\newdimen\CD@OA\newdimen\CD@PA\CD@tG\CD@gE
\CD@@A\CD@y\CD@tG\CD@hE\CD@EA\CD@BA\newdimen\CD@RA\newdimen\CD@SA\newcount
\CD@yA\newcount\CD@zA\newdimen\CD@QA\newbox\CD@DA\CD@tG\CD@lE\CD@dA\CD@bA
\newcount\CD@LH\newcount\CD@TC\def\CD@V#1#2{\ifdim#1<#2\relax#1=#2\relax\fi}%
\def\CD@X#1#2{\ifdim#1>#2\relax#1=#2\relax\fi}\newdimen\CD@XH\CD@XH=1sp
\newdimen\CD@zC\CD@zC\z@\def\CD@cJ{\ifdim\CD@zC=1em\else\CD@nJ\fi}\def\CD@nJ{%
\CD@zC1em\def\CD@NC{\fontdimen8\textfont3
}\CD@@J\CD@NJ\setbox0=\vbox{\CD@t
\noindent\CD@k\null\penalty-9993\null\CD@ND\null\endgraf\setbox0=\lastbox
\unskip\unpenalty\setbox1=\lastbox\global\setbox\CD@IG=\hbox{\unhbox0\unskip
\unskip\unpenalty\setbox0=\lastbox}\global\setbox\CD@KG=\hbox{\unhbox1\unskip
\unpenalty\setbox1=\lastbox}}}\newdimen\CD@@I\CD@@I=1true in
\divide\CD@@I300
\def\CD@zH#1{\multiply#1\tw@\advance#1\ifnum#1<\z@-\else+\fi\CD@@I\divide#1%
\tw@\divide#1\CD@@I\multiply#1\CD@@I}\def\MapBreadth{\afterassignment\CD@gI
\CD@LF}\newdimen\CD@LF\newdimen\CD@oI\def\CD@gI{\CD@oI\CD@LF\CD@V\CD@@I{4%
\CD@XH}\CD@X\CD@@I\p@\CD@zH\CD@oI\ifdim\CD@LF>\z@\CD@V\CD@oI\CD@@I\fi\CD@cJ}%
\def\CD@RJ#1{\CD@zD\count@\CD@@I#1\ifnum\count@>\z@\divide\CD@@I\count@\fi
\CD@gI\CD@NJ}\def\CD@NJ{\dimen@\CD@QC\count@\dimen@\divide\count@5\divide
\count@\CD@@I\edef\CD@OC{\the\count@}}\def\CD@AJ{\CD@QJ\z@}\def\CD@QJ#1{%
\CD@tI\axisheight\advance\CD@tI#1\relax\advance\CD@tI-.5\CD@oI\CD@zH\CD@tI
\CD@sI-\CD@tI\advance\CD@tI\CD@LF}\newdimen\CD@DC\CD@DC\z@\newdimen\CD@eJ
\CD@eJ\z@\def\CD@CJ#1{\CD@sI#1\relax\CD@tI\CD@sI\advance\CD@tI\CD@LF\relax}%
\def\horizhtdp{height\CD@tI depth\CD@sI}\def\axisheight{\fontdimen22\the
\textfont\tw@}\def\script@axisheight{\fontdimen22\the\scriptfont\tw@}\def
\ss@axisheight{\fontdimen22\the\scriptscriptfont\tw@}\def\CD@NC{0.4pt}\def
\CD@VK{\fontdimen3\textfont\z@}\def\CD@UK{\fontdimen3\textfont\z@}\newdimen
\PileSpacing\newdimen\CD@nA\CD@nA\z@\def\CD@RG{\ifincommdiag1.3em\else2em\fi}%
\newdimen\CD@YB\def\CellSize{\afterassignment\CD@kB\DiagramCellHeight}%
\newdimen\DiagramCellHeight\DiagramCellHeight-\maxdimen\newdimen
\DiagramCellWidth\DiagramCellWidth-\maxdimen\def\CD@kB{\DiagramCellWidth
\DiagramCellHeight}\def\CD@QC{3em}\newdimen\MapShortFall\def\MapsAbut{%
\MapShortFall\z@\objectheight\z@\objectwidth\z@}\newdimen\CD@iA\CD@iA\z@
\CD@tG\CD@vE\CD@aB\CD@ZB\expandafter\ifx\expandafter\iftrue\csname
ifUglyObsoleteDiagrams\endcsname\CD@ZB\else\CD@aB\fi\CD@nF{%
ifUglyObsoleteDiagrams}{relax}\newif\ifUglyObsoleteDiagrams\def\CD@nK{\CD@aB
\UglyObsoleteDiagramsfalse}\def\CD@oK{\CD@ZB\UglyObsoleteDiagramstrue}\CD@vE
\CD@nK\else\CD@oK\fi\CD@tG\CD@hK\CD@dK\CD@cK\CD@cK\def\CD@sK{\ifx\pdfoutput
\CD@qK\else\ifx\pdfoutput\relax\else\ifnum\pdfoutput>\z@\CD@pK\fi\fi\fi}
\def
\CD@pK{\global\CD@dK\global\CD@aB\global\UglyObsoleteDiagramsfalse\global\let
\CD@n\empty\global\let\CD@oK\relax\global\let\CD@pK\relax\global\let\CD@sK
\relax}\def\CD@tK#1{}\ifx\pdfliteral\CD@qK\else\ifx
\pdfliteral\relax\else\let\CD@tK\pdfliteral\fi\fi\ifx\XeTeXrevision\CD@qK
\else\ifx\XeTeXrevision\relax\else\ifdim\XeTeXrevision pt<.996pt
\expandafter \message{! XeTeX version \XeTeXrevision\space does not
support PDF literals, so diagonals will not
work!}\else\expandafter\message{RUNNING UNDER XETEX
\XeTeXrevision}\CD@pK\fi\fi\fi\CD@sK\def\newarrowhead{\CD@mG h\CD@BG\CD@GG>}%
\def\newarrowtail{\CD@mG t\CD@BG\CD@GG>}\def\newarrowmiddle{\CD@mG m\CD@BG
\hbox@maths\empty}\def\newarrowfiller{\CD@mG f\CD@bE\CD@MK-}\def\CD@mG#1#2#3#%
4#5#6#7#8#9{\CD@RC{r#1:#5}{#2{#6}}\CD@RC{l#1:#5}{#2{#7}}\CD@RC{d#1:#5}{#3{#8}%
}\CD@RC{u#1:#5}{#3{#9}}\CD@vC{-#1:#5}{\expandafter\noexpand\csname-#1:#4%
\endcsname\noexpand\CD@MC}\CD@vC{+#1:#5}{\expandafter\noexpand\csname+#1:#4%
\endcsname\noexpand\CD@MC}}\CD@ZA\CD@MC{\CD@eF\space diagonals are used unless
PostScript is set}\def\defaultarrowhead#1{\edef\CD@sJ{#1}\CD@@J}\def\CD@@J{%
\CD@IJ\CD@sJ<>ht\CD@IJ\CD@sJ<>th}\def\CD@IJ#1#2#3#4#5{\CD@HJ{r#4}{#3}{l#5}{#2%
}{r#4:#1}\CD@HJ{r#5}{#2}{l#4}{#3}{l#4:#1}\CD@HJ{d#4}{#3}{u#5}{#2}{d#4:#1}%
\CD@HJ{d#5}{#2}{u#4}{#3}{u#4:#1}}\def\CD@HJ#1#2#3#4#5{\begingroup\aftergroup
\CD@GJ\CD@L{#1+:#2}\CD@L{#1:#2}\CD@L{#3:#4}\CD@L{#5}\endgroup}\def\CD@GJ#1#2#%
3#4{\csname
newbox\endcsname#1\def#2{\copy#1}\def#3{\copy#1}\setbox#1=\box
\voidb@x}\def\CD@sJ{}\CD@@J\def\CD@GJ#1#2#3#4{\setbox#1=#4}\ifx\tenln
\nullfont\def\CD@sJ{vee}\else\let\CD@sJ\CD@eF\fi\def\CD@xF#1#2#3{\begingroup
\aftergroup\CD@wF\CD@L{#1#2:#3#3}\CD@L{#1#2:#3}\aftergroup\CD@yF\CD@L{#1#2:#3%
-#3}\CD@L{#1#2:#3}\endgroup}\def\CD@wF#1#2{\def#1{\hbox{\rlap{#2}\kern.4%
\CD@zC#2}}}\def\CD@yF#1#2{\def#1{\hbox{\rlap{#2}\kern.4\CD@zC#2\kern-.4\CD@zC
}}}\CD@xF lh>\CD@xF rt>\CD@xF rh<\CD@xF rt<\def\CD@yF#1#2{\def#1{\hbox{\kern-%
.4\CD@zC\rlap{#2}\kern.4\CD@zC#2}}}\CD@xF rh>\CD@xF lh<\CD@xF lt>\CD@xF lt<%
\def\CD@wF#1#2{\def#1{\vbox{\vbox to\z@{#2\vss}\nointerlineskip\kern.4\CD@zC#%
2}}}\def\CD@yF#1#2{\def#1{\vbox{\vbox to\z@{#2\vss}\nointerlineskip\kern.4%
\CD@zC#2\kern-.4\CD@zC}}}\CD@xF uh>\CD@xF dt>\CD@xF dh<\CD@xF dt<\def\CD@yF#1%
#2{\def#1{\vbox{\kern-.4\CD@zC\vbox to\z@{#2\vss}\nointerlineskip\kern.4%
\CD@zC#2}}}\CD@xF dh>\CD@xF ut>\CD@xF uh<\CD@xF ut<\def\CD@BG#1{\hbox{%
\mathsurround\z@\offinterlineskip\CD@k\mkern-1.5mu{#1}\mkern-1.5mu\CD@ND}}%
\def\hbox@maths#1{\hbox{\CD@k#1\CD@ND}}\def\CD@GG#1{\hbox to\CD@LF{\setbox0=%
\hbox{\offinterlineskip\mathsurround\z@\CD@k{#1}\CD@ND}\dimen0.5\wd0\advance
\dimen0-.5\CD@oI\CD@zH{\dimen0}\kern-\dimen0\unhbox0\hss}}\def\CD@sB#1{\hbox
to2\CD@LF{\hss\offinterlineskip\mathsurround\z@\CD@k{#1}\CD@ND\hss}}\def
\CD@vF#1{\hbox{\mathsurround\z@\CD@k{#1}\CD@ND}}\def\CD@bE#1{\hbox{\kern-.15%
\CD@zC\CD@k{#1}\CD@ND\kern-.15\CD@zC}}\def\CD@MK#1{\vbox{\offinterlineskip
\kern-.2ex\CD@GG{#1}\kern-.2ex}}\def\@fillh{\xleaders\vrule\horizhtdp}\def
\@fillv{\xleaders\hrule width\CD@LF}\CD@nF{rf:-}{@fillh}\CD@nF{lf:-}{@fillh}%
\CD@nF{df:-}{@fillv}\CD@nF{uf:-}{@fillv}\CD@nF{rh:}{null}\CD@nF{rm:}{null}%
\CD@nF{rt:}{null}\CD@nF{lh:}{null}\CD@nF{lm:}{null}\CD@nF{lt:}{null}\CD@nF{dh%
:}{null}\CD@nF{dm:}{null}\CD@nF{dt:}{null}\CD@nF{uh:}{null}\CD@nF{um:}{null}%
\CD@nF{ut:}{null}\CD@nF{+h:}{null}\CD@nF{+m:}{null}\CD@nF{+t:}{null}\CD@nF{-h%
:}{null}\CD@nF{-m:}{null}\CD@nF{-t:}{null}\def\CD@@D{\hbox{\vrule
height 1pt
depth-1pt width 1pt}}\CD@RC{rf:}{\CD@@D}\CD@nF{lf:}{rf:}\CD@nF{+f:}{rf:}%
\CD@RC{df:}{\CD@@D}\CD@nF{uf:}{df:}\CD@nF{-f:}{df:}\def\CD@BD{\CD@U\null
\CD@@D\null\CD@@D\null}\edef\CD@lG{\string\newarrow}\def\newarrow#1#2#3#4#5#6%
{\begingroup\edef\@name{#1}\edef\CD@oJ{#2}\edef\CD@iD{#3}\edef\CD@QG{#4}\edef
\CD@jD{#5}\edef\CD@LE{#6}\let\CD@HE\CD@sG\let\CD@FK\CD@BH\let\@x\CD@AH\ifx
\CD@oJ\CD@iD\let\CD@oJ\empty\fi\ifx\CD@LE\CD@jD\let\CD@LE\empty\fi\def\CD@LI{%
r}\def\CD@SF{l}\def\CD@IC{d}\def\CD@yJ{u}\def\CD@gH{+}\def\@m{-}\ifx\CD@iD
\CD@jD\ifx\CD@QG\CD@iD\let\CD@QG\empty\fi\ifx\CD@LE\empty\ifx\CD@iD\CD@aE\let
\@x\CD@yG\else\let\@x\CD@zG\fi\fi\else\edef\CD@a{\CD@iD\CD@oJ}\ifx\CD@a\empty
\ifx\CD@QG\CD@jD\let\CD@QG\empty\fi\fi\fi\ifmmode\aftergroup\CD@kG\else\CD@@A
\CD@oB rh{head\space\space}\CD@LE\CD@oB rf{filler}\CD@iD\CD@oB rm{middle}%
\CD@QG\ifx\CD@jD\CD@iD\else\CD@oB rf{filler}\CD@jD\fi\CD@oB
rt{tail\space
\space}\CD@oJ\CD@gE\CD@HE\CD@FK\@x\CD@nG l-2+2{lu}{nw}\NorthWest\CD@nG r+2+2{%
ru}{ne}\NorthEast\CD@nG l-2-2{ld}{sw}\SouthWest\CD@nG
r+2-2{rd}{se}\SouthEast
\else\aftergroup\CD@b\CD@L{r\@name}\fi\fi\endgroup}\def\CD@sG{\CD@vG\CD@LI
\CD@SF rl\Horizontal@Map}\def\CD@BH{\CD@vG\CD@IC\CD@yJ
du\Vertical@Map}\def
\CD@AH{\CD@vG\CD@gH\@m+-\Vector@Map}\def\CD@yG{\CD@vG\CD@gH\@m+-\Slant@Map}%
\def\CD@zG{\CD@vG\CD@gH\@m+-\Slope@Map}\catcode`\/=\active\def\CD@vG#1#2#3#4#%
5{\CD@jG#1#3#5t:\CD@oJ/f:\CD@iD/m:\CD@QG/f:\CD@jD/h:\CD@LE//\CD@jG#2#4#5h:%
\CD@LE/f:\CD@jD/m:\CD@QG/f:\CD@iD/t:\CD@oJ//}\def\CD@jG#1#2#3#4//{\edef\CD@fG
{#2}\aftergroup\sdef\CD@L{#1\@name}\aftergroup{\aftergroup#3\CD@M#4//%
\aftergroup}}\def\CD@M#1/{\edef\CD@EH{#1}\ifx\CD@EH\empty\else\CD@L{\CD@fG#1}%
\expandafter\CD@M\fi}\catcode`\/=12
\def\CD@nG#1#2#3#4#5#6#7#8{\aftergroup
\sdef\CD@L{#6\@name}\aftergroup{\CD@L{#2\@name}\if#2#4\aftergroup\CD@CI\else
\aftergroup\CD@BI\fi\CD@L{#1\@name}%
\aftergroup(\aftergroup#3\aftergroup,\aftergroup#5\aftergroup)\aftergroup}}%
\def\CD@oB#1#2#3#4{\expandafter\ifx\csname#1#2:#4\endcsname\relax\CD@y\CD@gB{%
arrow#3 "#4" undefined}\fi}\CD@rG\CD@VE{All five components must be
defined before an arrow.}\CD@rG\CD@SE{\CD@lG, unlike
\string\HorizontalMap, is a
declaration.}\def\CD@b#1{\CD@YA{Arrows \string#1 etc could not be defined}%
\CD@VE}\def\CD@kG{\CD@YA{misplaced \CD@lG}\CD@SE}\def\newdiagramgrid#1#2#3{%
\CD@RC{cdgh@#1}{#2,],}
\CD@RC{cdgv@#1}{#3,],}}
\CD@tG\ifincommdiag\incommdiagtrue\incommdiagfalse\CD@tG\CD@@F\CD@IF\CD@HF
\newcount\CD@VA\CD@VA=0 \def\CD@yH{\CD@VA6 }\def\CD@OB{\CD@VA1 \global\CD@yA1
\CD@DE\CD@YF\empty}\def\CD@YF{}\def\CD@nB#1{\relax\CD@MD\edef\CD@vJ{#1}%
\begingroup\CD@rE\else\ifcase\CD@VA\ifmmode\else\CD@YG\CD@E0\fi\or\CD@cE5\or
\CD@YG\CD@F5\or\CD@YG\CD@B5\or\CD@YG\CD@B5\or\CD@YG\CD@C5\or\CD@cE7\or\CD@YG
\CD@D7\fi\fi\endgroup\xdef\CD@YF{#1}}\def\CD@pB#1#2#3#4#5{\relax\CD@MD\xdef
\CD@vJ{#4}\begingroup\ifnum\CD@VA<#1
\expandafter\CD@cE\ifcase\CD@VA0\or#2\or
#3\else#2\fi\else\ifnum\CD@VA<6
\CD@tJ\CD@YG\CD@B#2\else\CD@YG\CD@G#2\fi\fi
\endgroup\CD@DE\CD@YF\CD@vJ\ifincommdiag\let\CD@ZD#5\else\let\CD@ZD\CD@LK\fi}%
\def\CD@yI{\global\CD@yA=\ifnum\CD@VA<5 1\else2\fi\relax}\def\CD@OI{\CD@VA
\CD@yA}\def\CD@cE#1{\aftergroup\CD@VA\aftergroup#1\aftergroup\relax}\def
\CD@HH{\def\CD@nB##1{\relax}\let\CD@pB\CD@FH\let\CD@yH\relax\let\CD@OB\relax
\let\CD@yI\relax\let\CD@OI\relax}\def\CD@FH#1#2#3#4#5{\ifincommdiag\let\CD@ZD
#5\else\xdef\CD@vJ{#4}\let\CD@ZD\CD@LK\fi}\def\CD@YG#1{\aftergroup#1%
\aftergroup\relax\CD@cE}\def\CD@B{\CD@YE\CD@S\CD@ME\CD@Q}\def\CD@G{\CD@YE{%
\CD@yC\CD@S}\CD@XE\CD@QD\CD@Q}\def\CD@F{\CD@YE{*\CD@S}\CD@RE\clubsuit\CD@Q}%
\def\CD@C{\CD@YE{\CD@S*\CD@S}\CD@RE\CD@Q\clubsuit\CD@Q}\def\CD@D{\CD@YE\CD@EC
\CD@TE\\}\def\CD@E{\CD@YE\CD@nC\CD@QE\CD@k}\def\CD@LK{\CD@YA{\CD@vJ\space
ignored \CD@dH}\CD@WE}\def\CD@FE{}\def\CD@d{\CD@YA{maps must never
be enclosed in braces}\CD@OE}\def\CD@dH{outside
diagram}\def\CD@FC{\string\HonV, \string \VonH\space and
\string\HmeetV}\CD@rG\CD@ME{The way that horizontal and vertical
arrows are terminated implicitly means\CD@uG that they cannot be
mixed with each other or with
\CD@FC.}\CD@rG\CD@XE{\string\pile\space is for
parallel horizontal arrows; verticals can just be put together in\CD@uG a cell%
. \CD@FC\space are not meaningful in a
\string\pile.}\CD@rG\CD@RE{The horizontal maps must point to an
object, not each other (I've put in\CD@uG one which you're unlikely
to want). Use \string\pile\space if you want them
parallel.}\CD@rG\CD@TE{Parallel horizontal arrows must be in
separate layers of a \string\pile.}\CD@rG\CD@QE{Horizontal arrows
may be used \CD@dH s, but
must still be in maths.}\CD@rG\CD@WE{Vertical arrows, \CD@FC\space\CD@dH s don%
't know where\CD@uG where to terminate.}\CD@rG\CD@OE{This prevents
them from stretching correctly.}\def\CD@YE#1{\CD@YA{"#1" inserted
\ifx\CD@YF\empty before \CD@vJ\else between \CD@YF\ifx\CD@YF\CD@vJ
s\else\space and \CD@vJ\fi \fi}}\count@=\year\multiply\count@12
\advance\count@\month\ifnum\count@>24247 \expandafter\endinput\fi\def
\Horizontal@Map{\CD@nB{horizontal map}\CD@LC\CD@TJ\CD@qD}\def\CD@TJ{\CD@GB-%
9999 \let\CD@ZD\CD@XD\ifincommdiag\else\CD@cJ\ifinpile\else\skip2\z@ plus 1.5%
\CD@VK minus .5\CD@UK\skip4\skip2 \fi\fi\let\CD@kD\@fillh\CD@nF{fill@dot}{rf:%
.}}\def\Vector@Map{\CD@HK4}\def\Slant@Map{\CD@HK{\CD@EF255\else6\fi}}\def
\Slope@Map{\CD@HK\CD@OC}\def\CD@HK#1#2#3#4#5#6{\CD@LC\def\CD@WK{2}\def\CD@aK{%
2}\def\CD@ZK{1}\def\CD@bK{1}\let\Horizontal@Map\CD@nI\def\CD@OG{#1}\def\CD@NI
{\CD@U#2#3#4#5#6}}\def\CD@nI{\CD@TJ\CD@JB\let\CD@ZD\CD@TD\CD@qD}\CD@tG\CD@pE
\CD@rA\CD@qA\CD@rA\def\cds@missives{\CD@rA}\def\CD@TD{\CD@vE\let\CD@OG\CD@OC
\CD@x\CD@zE\CD@WF\fi\setbox0\hbox{\incommdiagfalse\CD@HI}\CD@pE\CD@aD\else
\global\CD@YC\CD@bD\fi\ifvoid6 \ifvoid7
\CD@eE\fi\fi\CD@zE\else\CD@BD\global
\CD@YC\let\CD@CG\CD@IH\CD@YD\fi\else\CD@NI\CD@MI\global\CD@YC\CD@YD\fi}\def
\CD@LC{\begingroup\dimen1=\MapShortFall\dimen2=\CD@RG\dimen5=\MapShortFall
\setbox3=\box\voidb@x\setbox6=\box\voidb@x\setbox7=\box\voidb@x\CD@pD
\mathsurround\z@\skip2\z@ plus1fill minus 1000pt\skip4\skip2
\CD@TB}\CD@tG
\CD@tE\CD@UB\CD@TB\def\CD@U#1#2#3#4#5{\let\CD@oJ#1\let\CD@iD#2\let\CD@QG#3%
\let\CD@jD#4\let\CD@LE#5\CD@TB\ifx\CD@iD\CD@jD\CD@UB\fi}\def\CD@qD#1#2#3#4#5{%
\CD@U#1#2#3#4#5\CD@tD}\def\Vertical@Map{\CD@pB433{vertical
map}\CD@cD\CD@LC
\CD@GB-9995 \let\CD@kD\@fillv\CD@nF{fill@dot}{df:.}\CD@qD}\def\break@args{%
\def\CD@tD{\CD@ZD}\CD@ZD\endgroup\aftergroup\CD@FE}\def\CD@MJ{\setbox1=\CD@oJ
\setbox5=\CD@LE\ifvoid3 \ifx\CD@QG\null\else\setbox3=\CD@QG\fi\fi\CD@@G2%
\CD@iD\CD@@G4\CD@jD}\def\CD@pF#1{\ifvoid1\else\CD@oF1#1\fi\ifvoid2\else\CD@oF
2#1\fi\ifvoid3\else\CD@oF3#1\fi\ifvoid4\else\CD@oF4#1\fi\ifvoid5\else\CD@oF5#%
1\fi}
\def\CD@oF#1#2{\setbox#1\vbox{\offinterlineskip\box#1\dimen@\prevdepth
\advance\dimen@-#2\relax\setbox0\null\dp0\dimen@\ht0-\dimen@\box0}}\def\CD@@G
#1#2{\ifx#2\CD@kD\setbox#1=\box\voidb@x\else\setbox#1=#2\def#2{\xleaders\box#%
1}\fi}\CD@ZA\CD@BK{\string\HorizontalMap, \string\VerticalMap\space
and
\string\DiagonalMap\CD@uG are obsolete - use \CD@lG\space to pre-define maps}%
\def\HorizontalMap#1#2#3#4#5{\CD@BK\CD@nB{old horizontal map}\CD@LC\CD@TJ\def
\CD@oJ{\CD@UH{#1}}\CD@SH\CD@iD{#2}\def\CD@QG{\CD@UH{#3}}\CD@SH\CD@jD{#4}\def
\CD@LE{\CD@UH{#5}}\CD@tD}\def\VerticalMap#1#2#3#4#5{\CD@BK\CD@pB433{vertical
map}\CD@cD\CD@LC\CD@GB-9995
\let\CD@kD\@fillv\def\CD@oJ{\CD@GG{#1}}\CD@VH
\CD@iD{#2}\def\CD@QG{\CD@GG{#3}}\CD@VH\CD@jD{#4}\def\CD@LE{\CD@GG{#5}}\CD@tD}%
\def\DiagonalMap#1#2#3#4#5{\CD@BK\CD@LC\def\CD@OG{4}\let\CD@kD\CD@qK\let
\CD@ZD\CD@YD\def\CD@WK{2}\def\CD@aK{2}\def\CD@ZK{1}\def\CD@bK{1}\def\CD@QG{%
\CD@vF{#3}}\ifPositiveGradient\let\mv\raise\def\CD@oJ{\CD@vF{#5}}\def\CD@iD{%
\CD@vF{#4}}\def\CD@jD{\CD@vF{#2}}\def\CD@LE{\CD@vF{#1}}\else\let\mv\lower\def
\CD@oJ{\CD@vF{#1}}\def\CD@iD{\CD@vF{#2}}\def\CD@jD{\CD@vF{#4}}\def\CD@LE{%
\CD@vF{#5}}\fi\CD@tD}\def\CD@aE{-}\def\CD@AD{\empty}\def\CD@SH{\CD@EG\CD@bE
\CD@aE\@fillh}\def\CD@VH{\CD@EG\CD@MK\CD@KK\@fillv}\def\CD@EG#1#2#3#4#5{\def
\CD@CH{#5}\ifx\CD@CH#2\let#4#3\else\let#4\null\ifx\CD@CH\empty\else\ifx\CD@CH
\CD@AD\else\let#4\CD@CH\fi\fi\fi}\def\CD@UH#1{\hbox{\mathsurround\z@
\offinterlineskip\def\CD@CH{#1}\ifx\CD@CH\empty\else\ifx\CD@CH\CD@AD\else
\CD@k\mkern-1.5mu{\CD@CH}\mkern-1.5mu\CD@ND\fi\fi}}\def\CD@yD#1#2{\setbox#1=%
\hbox\bgroup\setbox0=\hbox{\CD@k\labelstyle()\CD@ND}
\setbox1=\null\ht1\ht0\dp1\dp0\box1
\kern.1\CD@zC\CD@k\bgroup\labelstyle
\aftergroup\CD@LD\CD@xD}\def\CD@LD{\CD@ND\kern.1\CD@zC\egroup\CD@tD}\def
\CD@xD{\futurelet\CD@EH\CD@mJ}\def\CD@mJ{
\catcase\bgroup:\CD@v;\catcase\egroup:\missing@label;\catcase\space:\CD@TF;%
\tokcase[:\CD@XF;
\default:\CD@zJ;\endswitch}\def\CD@v{\let\CD@MD\CD@c\let\CD@CH}\def\CD@zJ#1{%
\let\CD@UF\egroup{\let\actually@braces@missing@around@macro@in@label\CD@ZH
\let\CD@MD\CD@xC\let\CD@UF\CD@VF#1%
\actually@braces@missing@around@macro@in@label}\CD@UF}\def
\actually@braces@missing@around@macro@in@label{\let\CD@CH=}\def\missing@label
{\egroup\CD@YA{missing
label}\CD@PE}\def\CD@xC{\egroup\missing@label}\outer
\def\CD@ZH{}\def\CD@UF{}\def\CD@VF{\CD@wC\CD@UF}\def\CD@MD{}\def\CD@XF{\let
\CD@N\CD@xD\get@square@arg\CD@AE}\CD@rG\CD@PE{The text which has
just been read is not allowed within map
labels.}\def\CD@c{\egroup\CD@YA{missing \CD@yC \space inserted after
label}\CD@PE}\def\upper@label{\CD@oD\CD@yD6}\def
\lower@label{\def\positional@{\CD@@A\break@args}\CD@yD7}\def\middle@label{%
\CD@yD3}\CD@tG\CD@yE\CD@pD\CD@oD\def\CD@iF{\ifPositiveGradient\CD@tJ
\expandafter\upper@label\else\expandafter\lower@label\fi}\def\CD@iI{%
\ifPositiveGradient\CD@tJ\expandafter\lower@label\else\expandafter
\upper@label\fi}\def\positional@{\CD@gB{labels as positional
arguments are
obsolete}\CD@yE\CD@tJ\expandafter\upper@label\else\expandafter\lower@label\fi
-}\def\CD@tD{\futurelet\CD@EH\switch@arg}\def\eat@space{\afterassignment
\CD@tD\let\CD@EH= }\def\CD@TF{\afterassignment\CD@xD\let\CD@EH= }\def\CD@BC{%
\get@round@pair\CD@uD}\def\CD@uD#1#2{\def\CD@WK{#1}\def\CD@aK{#2}\CD@tD}\def
\optional@{\let\CD@N\CD@tD\get@square@arg\CD@AE}\def\CD@JJ.{\CD@sC\CD@tD}\def
\CD@sC{\let\CD@iD\fill@dot\let\CD@jD\fill@dot\def\CD@MI{\let\CD@iD\dfdot\let
\CD@jD\dfdot}}\def\CD@MI{}\def\CD@@E#1,{\CD@nH#1,\begingroup\ifx\@name\CD@RD
\CD@FF\aftergroup\CD@e\fi\aftergroup\CD@jC\else\expandafter\def\expandafter
\CD@RF\expandafter{\csname\@name\endcsname}\expandafter\CD@vD\CD@RF\CD@KD\ifx
\CD@RF\empty\aftergroup\CD@pC\expandafter\aftergroup\csname\CD@FB\@name
\endcsname\expandafter\aftergroup\csname\CD@FB @\@name\endcsname\else\gdef
\CD@GE{#1}\CD@gB{\string\relax\space inserted before
`[\CD@GE'}\message{(I was
trying to read this as a \CD@tA\ option.)}\aftergroup\CD@H\fi\fi\endgroup}%
\def\CD@vD#1#2\CD@KD{\def\CD@RF{#2}}\def\CD@jC{\let\CD@CH\CD@N\let\CD@N\relax
\CD@CH}\def\CD@H#1],{
\CD@jC\relax\def\CD@RF{#1}\ifx\CD@RF\empty\def\CD@RF{[\CD@GE]}%
\else\def\CD@RF{[\CD@GE,#1]}
\fi\CD@RF}\def\CD@pC#1#2{\ifx#2\CD@qK\ifx#1\CD@qK\CD@gB{option
`\@name'
undefined}\else#1\fi\else\CD@FF\expandafter#2\CD@GK\CD@PK\else\CD@QK\fi\fi
\CD@DH}\CD@tG\CD@FF\CD@QK\CD@PK\def\CD@nH#1,{\CD@FF\ifx\CD@GK\CD@qK\CD@e\else
\expandafter\CD@oH\CD@GK,#1,(,),(,)[]%
\fi\fi\CD@FF\else\CD@mH#1==,\fi}\def\CD@e{\CD@gB{option `\@name'
needs (x,y)
value}\CD@PK\let\@name\empty}\def\CD@mH#1=#2=#3,{\def\@name{#1}\def\CD@GK{#2}%
\def\CD@RF{#3}\ifx\CD@RF\empty\let\CD@GK\CD@qK\fi}%
\def\CD@oH#1(#2,#3)#4,(#5,#6)#7[]{\def\CD@GK{{#2}{#3}}\def\CD@RF{#1#4#5#6}%
\ifx\CD@RF\empty\def\CD@RF{#7}\ifx\CD@RF\empty\CD@e\fi\else\CD@e\fi}\def
\CD@FB{cds@}\let\CD@N\relax\def\CD@zD#1{\ifx\CD@GK\CD@qK\CD@gB{option
`\@name
' needs a value}\else#1\CD@GK\relax\fi}\def\CD@BE#1#2{\ifx\CD@GK\CD@qK#1#2%
\relax\else#1\CD@GK\relax\fi}\def\cds@@showpair#1#2{\message{x=#1,y=#2}}\def
\cds@@diagonalbase#1#2{\edef\CD@ZK{#1}\edef\CD@bK{#2}}\def\CD@DI#1{\def\CD@CH
{#1}\CD@nF{@x}{cdps@#1}\ifx\CD@CH\empty\CD@f\CD@CH{cannot be
used}\else\ifx
\CD@CH\relax\CD@f\CD@CH{unknown}\else\let\CD@IK\@x\fi\fi}\def\CD@f#1#2{\CD@gB
{PostScript translator `#1' #2}}\def\CD@PH{}\def\CD@PJ{\CD@fA\edef\CD@PH{%
\noexpand\CD@KB{\@name\space ignored within
maths}}}\def\diagramstyle{\CD@cJ
\let\CD@N\relax\CD@CF\CD@AE\CD@AE}\CD@tG\CD@sE
\CD@SB\CD@RB\CD@tG\CD@qE\CD@EB\CD@DB\CD@tG\CD@oE\CD@pA\CD@oA\CD@tG\CD@iE
\CD@HA\CD@GA\CD@HA\CD@tG\CD@jE\CD@JA\CD@IA\CD@tG\CD@kE\CD@LA\CD@KA\CD@tG
\CD@EF\CD@DK\CD@CK\CD@tG\CD@rE\CD@JB\CD@IB\CD@tG\CD@mE\CD@gA\CD@fA\CD@tG
\CD@nE\CD@kA\CD@jA\CD@tG\CD@AF\CD@iG\CD@hG\CD@RC{cds@
}{}\CD@RC{cds@}{}\CD@RC
{cds@1em}{\CellSize1\CD@zC}\CD@RC{cds@1.5em}{\CellSize1.5\CD@zC}\CD@RC{cds@2%
em}{\CellSize2\CD@zC}\CD@RC{cds@2.5em}{\CellSize2.5\CD@zC}\CD@RC{cds@3em}{%
\CellSize3\CD@zC}\CD@RC{cds@3.5em}{\CellSize3.5\CD@zC}\CD@RC{cds@4em}{%
\CellSize4\CD@zC}\CD@RC{cds@4.5em}{\CellSize4.5\CD@zC}\CD@RC{cds@5em}{%
\CellSize5\CD@zC}\CD@RC{cds@6em}{\CellSize6\CD@zC}\CD@RC{cds@7em}{\CellSize7%
\CD@zC}\CD@RC{cds@8em}{\CellSize8\CD@zC}\def\cds@abut{\MapsAbut\dimen1\z@
\dimen5\z@}\def\cds@alignlabels{\CD@IA\CD@KA}\def\cds@amstex{\ifincommdiag
\CD@O\else\def\CD{\diagram[amstex]}
\fi\CD@T\catcode`\@\active}\def\cds@b{\let\CD@dB\CD@bB}\def\cds@balance{\let
\CD@hA\CD@AA}\let\cds@bottom\cds@b\def\cds@center{\cds@vcentre\cds@nobalance}%
\let\cds@centre\cds@center\def\cds@centerdisplay{\CD@HA\CD@PJ\cds@balance}%
\let\cds@centredisplay\cds@centerdisplay\def\cds@crab{\CD@BE\CD@DC{.5%
\PileSpacing}}\CD@RC{cds@crab-}{\CD@DC-.5\PileSpacing}\CD@RC{cds@crab+}{%
\CD@DC.5\PileSpacing}\CD@RC{cds@crab++}{\CD@DC1.5\PileSpacing}\CD@RC{cds@crab%
--}{\CD@DC-1.5\PileSpacing}\def\cds@defaultsize{\CD@BE{\let\CD@QC}{3em}\CD@NJ
}\def\cds@displayoneliner{\CD@DB}\let\cds@dotted\CD@sC\def\cds@dpi{\CD@RJ{1%
truein}}\def\cds@dpm{\CD@RJ{100truecm}}\let\CD@XA\CD@qK\def\cds@eqno{\let
\CD@XA\CD@GK\let\CD@EJ\empty}\def\cds@fixed{\CD@qA}\CD@tG\CD@fE\CD@J\CD@I\def
\cds@flushleft{\CD@I\CD@GA\CD@PJ\cds@nobalance\CD@BE\CD@nA\CD@nA}\def\cds@gap
{\CD@AJ\setbox3=\null\ht3=\CD@tI\dp3=\CD@sI\CD@BE{\wd3=}\MapShortFall}
\def
\cds@grid{\ifx\CD@GK\CD@qK\let\h@grid\relax\let\v@grid\relax\else\CD@nF{%
h@grid}{cdgh@\CD@GK}\CD@nF{v@grid}{cdgv@\CD@GK}\ifx\h@grid\relax\CD@gB{%
unknown grid
`\CD@GK'}\else\CD@WB\fi\fi}\let\h@grid\relax\let\v@grid\relax
\def\cds@gridx{\ifx\CD@GK\CD@qK\else\cds@grid\fi\let\CD@CH\h@grid\let\h@grid
\v@grid\let\v@grid\CD@CH}\def\cds@h{\CD@zD\DiagramCellHeight}\def\cds@hcenter
{\let\CD@hA\CD@aA}\let\cds@hcentre\cds@hcenter\def\cds@heads{\CD@BE{\let
\CD@sJ}\CD@sJ\CD@@J\CD@vE\else\ifx\CD@sJ\CD@eF\else\CD@MC\fi\fi}\let
\cds@height\cds@h\let\cds@hmiddle\cds@balance\def\cds@htriangleheight{\CD@BE
\DiagramCellHeight\DiagramCellHeight\DiagramCellWidth1.73205%
\DiagramCellHeight}\def\cds@htrianglewidth{\CD@BE\DiagramCellWidth
\DiagramCellWidth\DiagramCellHeight.57735\DiagramCellWidth}\CD@tG\CD@zE\CD@eE
\CD@dE\CD@eE\def\cds@hug{\CD@eE}
\def\cds@inline{\CD@gA\let\CD@PH\empty}\def
\cds@inlineoneliner{\CD@EB}\CD@RC{cds@l>}{\CD@zD{\let\CD@RG}\dimen2=\CD@RG}%
\def\cds@labelstyle{\CD@zD{\let\labelstyle}}\def\cds@landscape{\CD@kA}\def
\cds@large{\CellSize5\CD@zC}\let\CD@EJ\empty\def\CD@FJ{\refstepcounter{%
equation}\def\CD@XA{\hbox{\@eqnnum}}}\def\cds@LaTeXeqno{\let\CD@EJ\CD@FJ}\def
\cds@lefteqno{\CD@pA}\def\cds@leftflush{\cds@flushleft\CD@J}\def
\cds@leftshortfall{\CD@zD{\dimen1 }}\def\cds@lowershortfall{%
\ifPositiveGradient\cds@leftshortfall\else\cds@rightshortfall\fi}\def
\cds@loose{\CD@VB}\def\cds@midhshaft{\CD@JA}\def\cds@midshaft{\CD@JA}\def
\cds@midvshaft{\CD@LA}\def\cds@moreoptions{\CD@@A}\let\cds@nobalance
\cds@hcenter\def\cds@nohcheck{\CD@HH}\def\cds@nohug{\CD@dE} \def
\cds@nooptions{\def\CD@aC{\CD@WD}}\let\cds@noorigin\cds@nobalance\def
\cds@nopixel{\CD@@I4\CD@XH\CD@cJ}\def\cds@UO{\CD@oK\global\let\CD@n\empty}%
\def\cds@UglyObsolete{\cds@UO\let\cds@PS\empty}\def\CD@rK#1{\CD@gB{option `#1%
' renamed as
`UglyObsolete'}}\def\cds@noPostScript{\CD@rK{noPostScript}}\def
\cds@noPS{\CD@rK{noPostScript}}\def\cds@notextflow{\CD@RB}\def\cds@noTPIC{%
\CD@CK}\def\cds@objectstyle{\CD@zD{\let\objectstyle}}\def\cds@origin{\let
\CD@hA\CD@iB}\def\cds@p{\CD@zD\PileSpacing}\let\cds@pilespacing\cds@p\def
\cds@pixelsize{\CD@zD\CD@@I\CD@gI}\def\cds@portrait{\CD@jA}\def
\cds@PostScript{\CD@nK\global\let\CD@n\empty\CD@BE\CD@DI\empty}\def\cds@PS{%
\CD@nK\global\let\CD@n\empty}\CD@GF\CD@n{\typeout{\CD@tA: try the
PostScript option for better
results}}\def\cds@repositionpullbacks{\let\make@pbk\CD@fH
\let\CD@qH\CD@pH}\def\cds@righteqno{\CD@oA}\def\cds@rightshortfall{\CD@zD{%
\dimen5
}}\def\cds@ruleaxis{\CD@zD{\let\axisheight}}\def\cds@cmex{\let\CD@GG
\CD@sB\let\CD@QJ\CD@CJ}\def\cds@s{\cds@height\DiagramCellWidth
\DiagramCellHeight}\def\cds@scriptlabels{\let\labelstyle\scriptstyle}\def
\cds@shortfall{\CD@zD\MapShortFall\dimen1\MapShortFall\dimen5\MapShortFall}%
\def\cds@showfirstpass{\CD@BE{\let\CD@nD}\z@}\def\cds@silent{\def\CD@KB##1{}%
\def\CD@gB##1{}}\let\cds@size\cds@s\def\cds@small{\CellSize2\CD@zC}\def
\cds@snake{\CD@BE\CD@eJ\z@}\def\cds@t{\let\CD@dB\CD@fB}\def\cds@textflow{%
\CD@SB\CD@PJ}\def\cds@thick{\let\CD@rF\tenlnw\CD@LF\CD@NC\CD@BE\MapBreadth{2%
\CD@LF}\CD@@J}\def\cds@thin{\let\CD@rF\tenln\CD@BE\MapBreadth{\CD@NC}\CD@@J}%
\def\cds@tight{\CD@WB}\let\cds@top\cds@t\def\cds@TPIC{\CD@DK}\def
\cds@uppershortfall{\ifPositiveGradient\cds@rightshortfall\else
\cds@leftshortfall\fi}\def\cds@vcenter{\let\CD@dB\CD@cB}\let\cds@vcentre
\cds@vcenter\def\cds@vtriangleheight{\CD@BE\DiagramCellHeight
\DiagramCellHeight\DiagramCellWidth.577035\DiagramCellHeight}\def
\cds@vtrianglewidth{\CD@BE\DiagramCellWidth\DiagramCellWidth
\DiagramCellHeight1.73205\DiagramCellWidth}\def\cds@vmiddle{\let\CD@dB\CD@eB}%
\def\cds@w{\CD@zD\DiagramCellWidth}\let\cds@width\cds@w\def\diagram{\relax
\protect\CD@bC}\def\enddiagram{\protect\CD@SG}\def\CD@bC{\CD@g\CD@uI
\incommdiagtrue\edef\CD@wI{\the\CD@NB}\global\CD@NB\z@\boxmaxdepth\maxdimen
\everycr{}\CD@sK\everymath{}\everyhbox{}\ifx\pdfsyncstop\CD@qK\else
\pdfsyncstop\fi\CD@aC}\def\CD@aC{\CD@y\let\CD@N\CD@ZC\CD@CF\CD@AE\CD@WD}\def
\CD@ZC{\CD@gE\expandafter\CD@aC\else\expandafter\CD@WD\fi}\def\CD@WD{\let
\CD@EH\relax\CD@nE\CD@vE\else\CD@hK\else\CD@KB{landscape ignored
without
PostScript}\CD@jA\fi\fi\fi\CD@EJ\setbox2=\vbox\bgroup\CD@JF\CD@VD}\def\CD@cH{%
\CD@nE\CD@fB\else\CD@dB\fi\CD@hA\nointerlineskip\setbox0=\null\ht0-\CD@pI\dp0%
\CD@pI\wd0\CD@kI\box0
\global\CD@QA\CD@kF\global\CD@yA\CD@XB\ifx\CD@NK\CD@qK
\global\CD@RA\CD@kF\else\global\CD@RA\CD@NK\fi\egroup\CD@zF\CD@nE\setbox2=%
\hbox to\dp2{\vrule height\wd2 depth\CD@QA
width\z@\global\CD@QA\ht2\ht2\z@ \dp2\z@\wd2\z@\CD@hK\CD@tK{q 0 1 -1
0 0 0 cm}\else\global\CD@iG\CD@IK{0 1
bturn}\fi\box2\CD@gK\hss}\CD@DB\fi\ifnum\CD@yA=1
\else\CD@DB\fi\global
\@ignorefalse\CD@mE\leavevmode\fi\ifvmode\CD@TA\else\ifmmode\CD@PH\CD@GI\else
\CD@qE\CD@gA\fi\ifinner\CD@gA\fi\CD@mE\CD@GI\else\CD@sE\CD@QB\else\CD@TA\fi
\fi\fi\fi\CD@dD}\def\CD@dD{\global\CD@NB\CD@wI\relax\CD@xE\global\CD@ID\else
\aftergroup\CD@mC\fi\if@ignore\aftergroup\ignorespaces\fi\CD@wC\ignorespaces}%
\def\CD@fB{\advance\CD@pI\dimen1\relax}\def\CD@eB{\advance\CD@pI.5\dimen1%
\relax}\def\CD@bB{}\def\CD@cB{\CD@fB\advance\CD@pI\CD@YB\divide\CD@pI2
\advance\CD@pI-\axisheight\relax}\def\CD@aA{}\def\CD@iB{\CD@kF\z@}\def\CD@AA{%
\ifdim\dimen2>\CD@kF\CD@kF\dimen2 \else\dimen2\CD@kF\CD@kI\dimen0
\advance
\CD@kI\dimen2 \fi}\def\CD@QB{\skip0\z@\relax\loop\skip1\lastskip\ifdim\skip1>%
\z@\unskip\advance\skip0\skip1
\repeat\vadjust{\prevdepth\dp\strutbox\penalty
\predisplaypenalty\vskip\abovedisplayskip\CD@UA\penalty\postdisplaypenalty
\vskip\belowdisplayskip}\ifdim\skip0=\z@\else\hskip\skip0
\global\@ignoretrue
\fi}\def\CD@TA{\CD@LG\kern-\displayindent\CD@UA\CD@LG\global\@ignoretrue}\def
\CD@UA{\hbox
to\hsize{\CD@fE\ifdim\CD@RA=\z@\else\advance\CD@QA-\CD@RA\setbox
2=\hbox{\kern\CD@RA\box2}\fi\fi\setbox1=\hbox{\ifx\CD@XA\CD@qK\else\CD@k
\CD@XA\CD@ND\fi}\CD@oE\CD@iE\else\advance\CD@QA\wd1 \fi\wd1\z@\box1
\fi\dimen 0\wd2 \advance\dimen0\wd1
\advance\dimen0-\hsize\ifdim\dimen0>-\CD@nA\CD@HA
\fi\advance\dimen0\CD@QA\ifdim\dimen0>\z@\CD@KB{wider than the page
by \the \dimen0
}\CD@HA\fi\CD@iE\hss\else\CD@V\CD@QA\CD@nA\fi\CD@GI\hss\kern-\wd1\box
1 }}\def\CD@GI{\CD@AF\CD@@F\else\CD@SC\global\CD@hG\fi\fi\kern\CD@QA\box2 }%
\CD@tG\CD@wE\CD@YC\CD@XC\def\CD@JF{\CD@cJ\ifdim\DiagramCellHeight=-\maxdimen
\DiagramCellHeight\CD@QC\fi\ifdim\DiagramCellWidth=-\maxdimen
\DiagramCellWidth\CD@QC\fi\global\CD@XC\CD@IF\let\CD@FE\empty\let\CD@z\CD@Q
\let\overprint\CD@eH\let\CD@s\CD@rJ\let\enddiagram\CD@ED\let\\\CD@cC\let\par
\CD@jH\let\CD@MD\empty\let\switch@arg\CD@PB\let\shift\CD@iA\baselineskip
\DiagramCellHeight\lineskip\z@\lineskiplimit\z@\mathsurround\z@\tabskip\z@
\CD@OB}\def\CD@VD{\penalty-123
\begingroup\CD@jA\aftergroup\CD@K\halign
\bgroup\global\advance\CD@NB1 \vadjust{\penalty1}\global\CD@FA\z@\CD@OB\CD@j#%
#\CD@DD\CD@Q\CD@Q\CD@OI\CD@j##\CD@DD\cr}\def\CD@ED{\CD@MD\CD@GD\crcr\egroup
\global\CD@JD\endgroup}\def\CD@j{\global\advance\CD@FA1
\futurelet\CD@EH\CD@i }\def\CD@i{\ifx\CD@EH\CD@DD\CD@tJ\hskip1sp
plus 1fil \relax\let\CD@DD\relax
\CD@vI\else\hfil\CD@k\objectstyle\let\CD@FE\CD@d\fi}\def\CD@DD{\CD@MD\relax
\CD@yI\CD@vI\global\CD@QA\CD@iA\penalty-9993
\CD@ND\hfil\null\kern-2\CD@QA
\null}\def\CD@cC{\cr}\def\across#1{\span\omit\mscount=#1
\global\advance
\CD@FA\mscount\global\advance\CD@FA\m@ne\CD@sF\ifnum\mscount>2
\CD@fJ\repeat
\ignorespaces}\def\CD@fJ{\relax\span\omit\advance\mscount\m@ne}\def\CD@qJ{%
\ifincommdiag\ifx\CD@iD\@fillh\ifx\CD@jD\@fillh\ifdim\dimen3>\z@\else\ifdim
\dimen2>93\CD@@I\ifdim\dimen2>18\p@\ifdim\CD@LF>\z@\count@\CD@bJ\advance
\count@\m@ne\ifnum\count@<\z@\count@20\let\CD@aJ\CD@uJ\fi\xdef\CD@bJ{\the
\count@}\fi\fi\fi\fi\fi\fi\fi}\def\CD@cG#1{\vrule\horizhtdp
width#1\dimen@
\kern2\dimen@}\def\CD@uJ{\rlap{\dimen@\CD@@I\CD@V\dimen@{.182\p@}\CD@zH
\dimen@\advance\CD@tI\dimen@\CD@cG0\CD@cG0\CD@cG2\CD@cG6\CD@cG6\CD@cG2\CD@cG0%
\CD@cG0\CD@cG2\CD@cG6\CD@cG0\CD@cG0\CD@cG2\CD@cG2\CD@cG6\CD@cG0\CD@cG0\CD@cG2%
\CD@cG6\CD@cG2\CD@cG2\CD@cG0\CD@cG0}}\def\CD@bJ{10}\def\CD@aJ{}\def\CD@XD{%
\CD@gE\CD@TB\fi\CD@x\CD@WF\CD@HI}\def\CD@x{\CD@QJ\CD@DC\CD@MJ\ifdim\CD@DC=\z@
\else\CD@pF\CD@DC\fi\ifvoid3
\setbox3=\null\ht3\CD@tI\dp3\CD@sI\else\CD@V{\ht
3}\CD@tI\CD@V{\dp3}\CD@sI\fi\dimen3=.5\wd3
\ifdim\dimen3=\z@\CD@tE\else\dimen
3-\CD@XH\fi\else\CD@TB\fi\CD@V{\dimen2}{\wd7}\CD@V{\dimen2}{\wd6}\CD@qJ
\advance\dimen2-2\dimen3 \dimen4.5\dimen2 \dimen2\dimen4 \advance\dimen2%
\CD@eJ\advance\dimen4-\CD@eJ\advance\dimen2-\wd1
\advance\dimen4-\wd5 \ifvoid 2
\else\CD@V{\ht3}{\ht2}\CD@V{\dp3}{\dp2}\CD@V{\dimen2}{\wd2}\fi\ifvoid4
\else
\CD@V{\ht3}{\ht4}\CD@V{\dp3}{\dp4}\CD@V{\dimen4}{\wd4}\fi\advance\skip2\dimen
2 \advance\skip4\dimen4 \CD@tE\advance\skip2\skip4 \dimen0\dimen5
\advance
\dimen0\wd5 \skip3-\skip4 \advance\skip3-\dimen0 \let\CD@jD\empty\else\skip3%
\z@\relax\dimen0\z@\fi}\def\CD@WF{\offinterlineskip\lineskip.2\CD@zC\ifvoid6
\else\setbox3=\vbox{\hbox to2\dimen3{\hss\box6\hss}\box3}\fi\ifvoid7
\else \setbox3=\vtop{\box3 \hbox
to2\dimen3{\hss\box7\hss}}\fi}\def\CD@HI{\kern \dimen1 \box1
\CD@aJ\CD@iD\hskip\skip2 \kern\dimen0 \ifincommdiag\CD@jE
\penalty1\fi\kern\dimen3 \penalty\CD@GB\hskip\skip3
\null\kern-\dimen3 \else \hskip\skip3 \fi\box3 \CD@jD\hskip\skip4
\box5 \kern\dimen5}\def\CD@MF{\ifnum
\CD@LH>\CD@TC\CD@V{\dimen1}\objectheight\CD@V{\dimen5}\objectheight\else\CD@V
{\dimen1}\objectwidth\CD@V{\dimen5}\objectwidth\fi}\def\CD@Y{\begingroup
\ifdim\dimen7=\z@\kern\dimen8 \else\ifdim\dimen6=\z@\kern\dimen9 \else\dimen5%
\dimen6 \dimen6\dimen9 \CD@KJ\dimen4\dimen2
\CD@dG{\dimen4}\dimen6\dimen5 \dimen7\dimen8
\CD@KJ\CD@iC{\dimen2}\ifdim\dimen2<\dimen4 \kern\dimen2 \else
\kern\dimen4
\fi\fi\fi\endgroup}\def\CD@jJ{\CD@JI\setbox\z@\hbox{\lower
\axisheight\hbox to\dimen2{\CD@DF\ifPositiveGradient\dimen8\ht\CD@MH\dimen9%
\CD@mI\else\dimen8\dp3 \dimen9\dimen1 \fi\else\dimen8
\ifPositiveGradient
\objectheight\else\z@\fi\dimen9\objectwidth\fi\advance\dimen8
\ifPositiveGradient-\fi\axisheight\CD@Y\unhbox\z@\CD@DF\ifPositiveGradient
\dimen8\dp3 \dimen9\dimen0
\else\dimen8\ht\CD@MH\dimen9\CD@mF\fi\else\dimen8
\ifPositiveGradient\z@\else\objectheight\fi\dimen9\objectwidth\fi\advance
\dimen8
\ifPositiveGradient\else-\fi\axisheight\CD@Y}}}\def\CD@bD{\dimen6
\CD@aK\DiagramCellHeight\dimen7 \CD@WK\DiagramCellWidth\CD@jJ
\ifPositiveGradient\advance\dimen7-\CD@ZK\DiagramCellWidth\else\dimen7
\CD@ZK
\DiagramCellWidth\dimen6\z@\fi\advance\dimen6-\CD@bK\DiagramCellHeight\CD@mK
\setbox0=\rlap{\kern-\dimen7
\lower\dimen6\box\z@}\ht0\z@\dp0\z@\raise \axisheight\box0
}\def\CD@mK{\setbox0\hbox{\ht\z@\z@\dp\z@\z@\wd\z@\z@\CD@hK
\expandafter\CD@tK{q \CD@eK\space\CD@lK\space\CD@kK\space\CD@eK\space0 0 cm}%
\else\global\CD@iG\CD@eD{\the\CD@TC\space\ifPositiveGradient\else-\fi\the
\CD@LH\space
bturn}\fi\box\z@\CD@gK}}\def\CD@vB{\advance\CD@hF-\CD@mI\CD@wJ
\CD@hF\advance\CD@wJ\CD@hI\ifvoid\CD@sH\ifdim\CD@wJ<.1em\ifnum\CD@gD=\@m\else
\CD@aG h\CD@wJ<.1em:objects
overprint:\CD@FA\CD@gD\fi\fi\else\ifhbox\CD@sH
\CD@SK\else\CD@TK\fi\advance\CD@wJ\CD@mI\CD@bH{-\CD@mI}{\box\CD@sH}{\CD@wJ}%
\z@\fi\CD@hF-\CD@mF\CD@gD\CD@FA\CD@hI\z@}\def\CD@SK{\setbox\CD@sH=\hbox{%
\unhbox\CD@sH\unskip\unpenalty}\setbox\CD@tH=\hbox{\unhbox\CD@tH\unskip
\unpenalty}\setbox\CD@sH=\hbox
to\CD@wJ{\CD@OA\wd\CD@sH\unhbox\CD@sH\CD@PA
\lastkern\unkern\ifdim\CD@PA=\z@\CD@UB\advance\CD@OA-\wd\CD@tH\else\CD@TB\fi
\ifnum\lastpenalty=\z@\else\CD@JA\unpenalty\fi\kern\CD@PA\ifdim\CD@hF<\CD@OA
\CD@JA\fi\ifdim\CD@hI<\wd\CD@tH\CD@JA\fi\CD@jE\CD@hI\CD@wJ\advance\CD@hI-%
\CD@OA\advance\CD@hI\wd\CD@tH\ifdim\CD@hI<2\wd\CD@tH\CD@aG
h\CD@hI<2\wd\CD@tH
:arrow too short:\CD@FA\CD@gD\fi\divide\CD@hI\tw@\CD@hF\CD@wJ\advance\CD@hF-%
\CD@hI\fi\CD@tE\kern-\CD@hI\fi\hbox
to\CD@hI{\unhbox\CD@tH}\CD@HG}}\CD@tG
\ifinpile\inpiletrue\inpilefalse\inpilefalse\def\pile{\protect\CD@UJ\protect
\CD@uH}\def\CD@uH#1{\CD@l#1\CD@QD}\def\CD@UJ{\CD@nB{pile}\setbox0=\vtop
\bgroup\aftergroup\CD@lD\inpiletrue\let\CD@FE\empty\let\pile\CD@KF\let\CD@QD
\CD@PD\let\CD@GD\CD@FD\CD@yH\baselineskip.5\PileSpacing\lineskip.1\CD@zC
\relax\lineskiplimit\lineskip\mathsurround\z@\tabskip\z@\let\\\CD@wH}\def
\CD@l{\CD@DE\CD@YF\empty\halign\bgroup\hfil\CD@k\let\CD@FE\CD@d\let\\\CD@vH##%
\CD@MD\CD@ND\hfil\CD@Q\CD@R##\cr}\CD@rG\CD@NE{pile only allows one column.}%
\CD@rG\CD@UE{you left it
out!}\def\CD@R{\CD@QD\CD@Q\relax\CD@YA{missing \CD@yC \space
inserted after \string\pile}\CD@NE}\def\CD@PD{\CD@MD\crcr\egroup
\egroup}\def\CD@GD{\CD@MD}\def\CD@FD{\CD@MD\relax\CD@QD\CD@YA{missing
\CD@yC
\space inserted between \string\pile\space and \CD@HD}\CD@UE}\def\CD@QD{%
\CD@MD}\def\CD@lD{\vbox{\dimen1\dp0 \unvbox0 \setbox0=\lastbox\advance\dimen1%
\dp0 \nointerlineskip\box0 \nointerlineskip\setbox0=\null\dp0.5\dimen1\ht0-%
\dp0 \box0}\ifincommdiag\CD@tJ\penalty-9998 \fi\xdef\CD@YF{pile}}\def\CD@vH{%
\cr}\def\CD@wH{\noalign{\skip@\prevdepth\advance\skip@-\baselineskip
\prevdepth\skip@}}\def\CD@KF#1{#1}\def\CD@TK{\setbox\CD@sH=\vbox{\unvbox
\CD@sH\setbox1=\lastbox\setbox0=\box\voidb@x\CD@tF\setbox\CD@sH=\lastbox
\ifhbox\CD@sH\CD@rC\repeat\unvbox0
\global\CD@QA\CD@ZE}\CD@ZE\CD@QA}\def
\CD@rC{\CD@jE\setbox\CD@sH=\hbox{\unhbox\CD@sH\unskip\setbox\CD@sH=\lastbox
\unskip\unhbox\CD@sH}\ifdim\CD@wJ<\wd\CD@sH\CD@aG
h\CD@wJ<\wd\CD@sH:arrow in
pile too short:\CD@FA\CD@gD\else\setbox\CD@sH=\hbox to\CD@wJ{\unhbox\CD@sH}%
\fi\else\CD@gJ\fi\setbox0=\vbox{\box\CD@sH\nointerlineskip\ifvoid0
\CD@tJ\box 1 \else\vskip\skip0 \unvbox0
\fi}\skip0=\lastskip\unskip}\def\CD@gJ{\penalty7
\noindent\unhbox\CD@sH\unskip\setbox\CD@sH=\lastbox\unskip\unhbox\CD@sH
\endgraf\setbox\CD@tH=\lastbox\unskip\setbox\CD@tH=\hbox{\CD@JG\unhbox\CD@tH
\unskip\unskip\unpenalty}\ifcase\prevgraf\cd@shouldnt
P\or\ifdim\CD@wJ<\wd \CD@tH\CD@aG h\CD@wJ<\wd\CD@sH:object in pile
too wide:\CD@FA\CD@gD\setbox \CD@sH=\hbox
to\CD@wJ{\hss\unhbox\CD@tH\hss}\else\setbox\CD@sH=\hbox to\CD@wJ
{\hss\kern\CD@hF\unhbox\CD@tH\kern\CD@hI\hss}\fi\or\setbox\CD@sH=\lastbox
\unskip\CD@SK\else\cd@shouldnt
Q\fi\unskip\unpenalty}\def\CD@cD{\CD@MJ\ifvoid 3
\setbox3=\null\ht3\axisheight\dp3-\ht3
\dimen3.5\CD@LF\else\dimen4\dp3 \dimen3.5\wd3
\setbox3=\CD@GG{\box3}\dp3\dimen4 \ifdim\ht3=-\dp3 \else\CD@TB
\fi\fi\dimen0\dimen3
\advance\dimen0-.5\CD@LF\setbox0\null\ht0\ht3\dp0\dp3\wd 0\wd3
\ifvoid6\else\setbox6\hbox{\unhbox6\kern\dimen0\kern2pt}\dimen0\wd6
\fi
\ifvoid7\else\setbox7\hbox{\kern2pt\kern\dimen3\unhbox7}\dimen3\wd7
\fi \setbox3\hbox{\ifvoid6\else\kern-\dimen0\unhbox6\fi\unhbox3
\ifvoid7\else \unhbox7\kern-\dimen3\fi}\ht3\ht0\dp3\dp0\wd3\wd0
\CD@tE\dimen4=\ht\CD@MH \advance\dimen4\dp5 \advance\dimen4\dimen1
\let\CD@jD\empty\else\dimen4\ht3 \fi\setbox0\null\ht0\dimen4
\offinterlineskip\setbox8=\vbox spread2ex{\kern
\dimen5 \box1 \CD@iD\vfill\CD@tE\else\kern\CD@eJ\fi\box0}\ht8=\z@\setbox9=%
\vtop spread2ex{\kern-\ht3 \kern-\CD@eJ\box3 \CD@jD\vfill\box5 \kern\dimen1}%
\dp9=\z@\hskip\dimen0plus.0001fil \box9 \kern-\CD@LF\box8
\CD@kE\penalty2 \fi \CD@tE\penalty1
\fi\kern\PileSpacing\kern-\PileSpacing\kern-.5\CD@LF\penalty
\CD@GB\null\kern\dimen3}\def\CD@cI{\ifhbox\CD@VA\CD@KB{clashing
verticals}\ht \CD@MH.5\dp\CD@VA\dp\CD@MH-\ht5
\CD@yB\ht\CD@MH\z@\dp\CD@MH\z@\fi\dimen1\dp
\CD@VA\CD@xA\prevgraf\unvbox\CD@VA\CD@wA\lastpenalty\unpenalty\setbox\CD@VA=%
\null\setbox\CD@lI=\hbox{\CD@JG\unhbox\CD@lI\unskip\unpenalty\dimen0\lastkern
\unkern\unkern\unkern\kern\dimen0
\CD@HG}\setbox\CD@lF=\hbox{\unhbox\CD@lF
\dimen0\lastkern\unkern\unkern\global\CD@QA\lastkern\unkern\kern\dimen0 }%
\CD@tF\ifnum\CD@xA>4
\CD@zI\repeat\unskip\unskip\advance\CD@mF.5\wd\CD@VA
\advance\CD@mF\wd\CD@lF\advance\CD@mI.5\wd\CD@VA\advance\CD@mI\wd\CD@lI\ifnum
\CD@FA=\CD@lA\CD@OA.5\wd\CD@VA\edef\CD@NK{\the\CD@OA}\fi\setbox\CD@VA=\hbox{%
\kern-\CD@mF\box\CD@lF\unhbox\CD@VA\box\CD@lI\kern-\CD@mI\penalty\CD@wA
\penalty\CD@NB}\ht\CD@VA\dimen1 \dp\CD@VA\z@\wd\CD@VA\CD@tB\CD@vB}\def\CD@zI{%
\ifdim\wd\CD@lF<\CD@QA\setbox\CD@lF=\hbox
to\CD@QA{\CD@JG\unhbox\CD@lF}\fi
\advance\CD@xA\m@ne\setbox\CD@VA=\hbox{\box\CD@lF\unhbox\CD@VA}\unskip\setbox
\CD@lF=\lastbox\setbox\CD@lF=\hbox{\unhbox\CD@lF\unskip\unpenalty\dimen0%
\lastkern\unkern\unkern\global\CD@QA\lastkern\unkern\kern\dimen0
}}\def\CD@yB
{\dimen1\dp\CD@VA\ifhbox\CD@VA\CD@xB\else\CD@zB\fi\setbox\CD@VA=\vbox{%
\penalty\CD@NB}\dp\CD@VA-\dp\CD@MH\wd\CD@VA\CD@tB}\def\CD@zB{\unvbox\CD@VA
\CD@wA\lastpenalty\unpenalty\ifdim\dimen1<\ht\CD@MH\CD@aG v\dimen1<\ht\CD@MH:%
rows overprint:\CD@NB\CD@wA\fi}\def\CD@xB{\dimen0=\ht\CD@VA\setbox\CD@VA=%
\hbox\bgroup\advance\dimen1-\ht\CD@MH\unhbox\CD@VA\CD@xA\lastpenalty
\unpenalty\CD@wA\lastpenalty\unpenalty\global\CD@RA-\lastkern\unkern\setbox0=%
\lastbox\CD@tF\setbox\CD@VA=\hbox{\box0\unhbox\CD@VA}\setbox0=\lastbox\ifhbox
0
\CD@kJ\repeat\global\CD@SA-\lastkern\unkern\global\CD@QA\CD@JK\unhbox\CD@VA
\egroup\CD@JK\CD@QA\CD@bH{\CD@SA}{\box\CD@VA}{\CD@RA}{\dimen1}}\def\CD@kJ{%
\setbox0=\hbox to\wd0\bgroup\unhbox0
\unskip\unpenalty\dimen7\lastkern\unkern \ifnum\lastpenalty=1
\unpenalty\CD@UB\else\CD@TB\fi\ifnum\lastpenalty=2
\unpenalty\dimen2.5\dimen0\advance\dimen2-.5\dimen1\advance\dimen2-%
\axisheight\else\dimen2\z@\fi\setbox0=\lastbox\dimen6\lastkern\unkern\setbox1%
=\lastbox\setbox0=\vbox{\unvbox0 \CD@tE\kern-\dimen1
\else\ifdim\dimen2=\z@ \else\kern\dimen2 \fi\fi}\ifdim\dimen0<\ht0
\CD@aG v\dimen0<\ht0:upper part of
vertical too short:{\CD@tE\CD@NB\else\CD@wA\fi}\CD@xA\else\setbox0=\vbox to%
\dimen0{\unvbox0}\fi\setbox1=\vtop{\unvbox1}\ifdim\dimen1<\dp1
\CD@aG v\dimen
1<\dp1:lower part of vertical too short:\CD@NB\CD@wA\else\setbox1=\vtop to%
\dimen1{\ifdim\dimen2=\z@\else\kern-\dimen2 \fi\unvbox1 }\fi\box1
\kern\dimen
6 \box0 \kern\dimen7 \CD@HG\global\CD@QA\CD@JK\egroup\CD@JK\CD@QA\relax}%
\countdef\CD@u=14
\newcount\CD@CA\newcount\CD@XB\newcount\CD@NB\let\CD@LB
\insc@unt\newcount\CD@FA\newcount\CD@lA\let\CD@mA\CD@XB\newcount\CD@MB\CD@tG
\CD@DF\CD@bI\CD@aI\CD@aI\def\CD@nD{-1}\def\CD@K{\ifnum\CD@nD<\z@\else
\begingroup\scrollmode\showboxdepth\CD@nD\showboxbreadth\maxdimen\showlists
\endgroup\fi\CD@bI\CD@zF\CD@CA=\CD@u\advance\CD@CA1 \CD@XB=\CD@CA\ifnum\CD@NB
=1
\CD@JA\fi\advance\CD@XB\CD@NB\dimen1\z@\skip0\z@\count@=\insc@unt\advance
\count@\CD@u\divide\count@2 \ifnum\CD@XB>\count@\CD@KB{The diagram
has too
many rows! It can't be reformatted.}\else\CD@NG\CD@WI\fi\CD@cH}\def\CD@NG{%
\CD@NB\CD@CA\CD@uF\ifnum\CD@NB<\CD@XB\setbox\CD@NB\box\voidb@x\advance\CD@NB1%
\relax\repeat\CD@NB\CD@CA\skip\z@\z@\CD@uF\CD@GB\lastpenalty\unpenalty\ifnum
\CD@GB>\z@\CD@KE\repeat\ifnum\CD@GB=-123 \CD@tJ\unpenalty\else\cd@shouldnt D%
\fi\ifx\v@grid\relax\else\CD@NB\CD@XB\advance\CD@NB\m@ne\expandafter\CD@VJ
\v@grid\fi\CD@MB\CD@mA\CD@tB\z@\CD@XG\ifx\h@grid\relax\else\expandafter\CD@LJ
\h@grid\fi\count@\CD@XB\advance\count@\m@ne\CD@YB\ht\count@}\def\CD@KE{%
\ifcase\CD@GB\or\CD@MG\else\CD@uA-\lastpenalty\unpenalty\CD@vA\lastpenalty
\unpenalty\setbox0=\lastbox\CD@WG\fi\CD@wD}\def\CD@wD{\skip1\lastskip\unskip
\advance\skip0\skip1 \ifdim\skip1=\z@\else\expandafter\CD@wD\fi}\def\CD@MG{%
\setbox0=\lastbox\CD@pI\dp0
\advance\CD@pI\skip\z@\skip\z@\z@\advance\CD@NF
\CD@pI\CD@uE\ifnum\CD@NB>\CD@CA\CD@NF\DiagramCellHeight\CD@pI\CD@NF\advance
\CD@pI-\CD@qI\fi\fi\CD@qI\ht0
\CD@NF\CD@qI\setbox\CD@NB\hbox{\unhbox\CD@NB
\unhbox0}\dp\CD@NB\CD@pI\ht\CD@NB\CD@qI\advance\CD@NB1
}\def\CD@WG{\ifnum
\CD@uA<\z@\advance\CD@uA\CD@XB\ifnum\CD@uA<\CD@CA\CD@UG\else\CD@OA\dp\CD@uA
\CD@PA\ht\CD@uA\setbox\CD@uA\hbox{\box\z@\penalty\CD@vA\penalty\CD@GB\unhbox
\CD@uA}\dp\CD@uA\CD@OA\ht\CD@uA\CD@PA\fi\else\CD@UG\fi}\def\CD@UG{\CD@KB{%
diagonal goes outside diagram
(lost)}}\def\CD@fI{\advance\CD@uA\CD@XB\ifnum
\CD@uA<\CD@CA\CD@UG\else\ifnum\CD@uA=\CD@NB\CD@VG\else\ifnum\CD@uA>\CD@NB
\cd@shouldnt
M\else\CD@OA\dp\CD@uA\CD@PA\ht\CD@uA\setbox\CD@uA\hbox{\box\z@
\penalty\CD@vA\penalty\CD@GB\unhbox\CD@uA}\dp\CD@uA\CD@OA\ht\CD@uA\CD@PA\fi
\fi\fi}\def\CD@WI{\CD@t\CD@AJ\setbox\CD@PC=\hbox{\CD@k A\@super f\CD@lJ f%
\CD@ND}\CD@ZE\z@\CD@JK\z@\CD@kI\z@\CD@kF\z@\CD@NB=\CD@XB\CD@NF\z@\CD@uB\z@
\CD@uF\ifnum\CD@NB>\CD@CA\advance\CD@NB\m@ne\CD@qI\ht\CD@NB\CD@pI\dp\CD@NB
\advance\CD@NF\CD@qI\CD@rI\advance\CD@uB\CD@NF\CD@KC\CD@ZI\CD@w\ht\CD@NB
\CD@qI\dp\CD@NB\CD@pI\nointerlineskip\box\CD@NB\CD@NF\CD@pI\setbox\CD@NB\null
\ht\CD@NB\CD@uB\repeat\CD@wB\nointerlineskip\box\CD@NB\CD@gG\CD@ZE
\DiagramCellWidth{width}\CD@gG\CD@JK\DiagramCellHeight{height}\CD@VA\CD@LB
\advance\CD@VA-\CD@lA\advance\CD@VA\m@ne\advance\CD@VA\CD@mA\dimen0\wd\CD@VA
\CD@tI\axisheight\dimen1\CD@uB\advance\dimen1-\CD@YB\dimen2\CD@kI\advance
\dimen2-\dimen0
\advance\CD@XB-\CD@CA\advance\CD@LB-\CD@lA}\count@\year
\multiply\count@12 \advance\count@\month\ifnum\count@>24254
\loop\iftrue \repeat\fi\def\CD@wB{\CD@qI-\CD@NF\CD@pI\CD@NF
\setbox\CD@MH=\null\dp\CD@MH\CD@NF\ht\CD@MH-\CD@NF\CD@mF\z@\CD@mI\z@\CD@lA
\CD@LB\advance\CD@lA-\CD@MB\advance\CD@lA\CD@mA\CD@FA\CD@LB\CD@VA\CD@MB\CD@sF
\ifnum\CD@FA>\CD@lA\advance\CD@FA\m@ne\advance\CD@VA\m@ne\CD@tB\wd\CD@VA
\setbox\CD@FA=\box\voidb@x\CD@yB\repeat\CD@w\ht\CD@NB\CD@qI\dp\CD@NB\CD@pI}%
\def\CD@gG#1#2#3{\ifdim#1>.01\CD@zC\CD@PA#2\relax\advance\CD@PA#1\relax
\advance\CD@PA.99\CD@zC\count@\CD@PA\divide\count@\CD@zC\CD@KB{increase cell #%
3 to \the\count@ em}\fi}\def\CD@rI{\CD@FA=\CD@LB\penalty4
\noindent\unhbox \CD@NB\CD@sF\unskip\setbox0=\lastbox\ifhbox0
\advance\CD@FA\m@ne\setbox\CD@FA \hbox
to\wd0{\null\penalty-9990\null\unhbox0}\repeat\CD@lA\CD@FA\advance
\CD@FA\CD@MB\advance\CD@FA-\CD@mA\ifnum\CD@FA<\CD@LB\count@\CD@FA\advance
\count@\m@ne\dimen0=\wd\count@\count@\CD@MB\advance\count@\m@ne\CD@tB\wd
\count@\CD@sF\ifnum\CD@FA<\CD@LB\CD@DJ\CD@XG\dimen0\wd\CD@FA\advance\CD@FA1
\repeat\fi\CD@sF\CD@GB\lastpenalty\unpenalty\ifnum\CD@GB>\z@\CD@vA
\lastpenalty\unpenalty\CD@VG\repeat\endgraf\unskip\ifnum\lastpenalty=4
\unpenalty\else\cd@shouldnt
S\fi}\def\CD@VG{\advance\CD@vA\CD@lA\advance
\CD@vA\m@ne\setbox0=\lastbox\ifnum\CD@vA<\CD@LB\setbox\CD@vA\hbox{\box0%
\penalty\CD@GB\unhbox\CD@vA}\else\CD@UG\fi}\def\CD@bG{}\CD@tG\CD@uE\CD@WB
\CD@VB\def\CD@DJ{\advance\dimen0\wd\CD@FA\divide\dimen0\tw@\CD@uE\dimen0%
\DiagramCellWidth\else\CD@V{\dimen0}\DiagramCellWidth\CD@pJ\fi\advance\CD@tB
\dimen0
}\def\CD@XG{\setbox\CD@MB=\vbox{}\dp\CD@MB=\CD@uB\wd\CD@MB\CD@tB
\advance\CD@MB1
}\def\CD@LJ#1,{\def\CD@GK{#1}\ifx\CD@GK\CD@RD\else\advance
\CD@tB\CD@GK\DiagramCellWidth\CD@XG\expandafter\CD@LJ\fi}\def\CD@VJ#1,{\def
\CD@GK{#1}\ifx\CD@GK\CD@RD\else\ifnum\CD@NB>\CD@CA\CD@NF\CD@GK
\DiagramCellHeight\advance\CD@NF-\dp\CD@NB\advance\CD@NB\m@ne\ht\CD@NB\CD@NF
\fi\expandafter\CD@VJ\fi}\def\CD@pJ{\CD@wE\CD@OA\dimen0 \advance\CD@OA-%
\DiagramCellWidth\ifdim\CD@OA>2\MapShortFall\CD@KB{badly drawn
diagonals (see
manual)}\let\CD@pJ\empty\fi\else\let\CD@pJ\empty\fi}\def\CD@KC{\CD@VA\CD@mA
\CD@sF\ifnum\CD@VA<\CD@MB\dimen0\dp\CD@VA\advance\dimen0\CD@NF\dp\CD@VA\dimen
0 \advance\CD@VA1 \repeat}\def\CD@bH#1#2#3#4{\ifnum\CD@FA<\CD@LB\CD@OA=#1%
\relax\setbox\CD@FA=\hbox{\setbox0=#2\dimen7=#4\relax\dimen8=#3\relax\ifhbox
\CD@FA\unhbox\CD@FA\advance\CD@OA-\lastkern\unkern\fi\ifdim\CD@OA=\z@\else
\kern-\CD@OA\fi\raise\dimen7\box0 \kern-\dimen8
}\ifnum\CD@FA=\CD@lA\CD@V \CD@kF\CD@OA\fi\else\cd@shouldnt
O\fi}\def\CD@w{\setbox\CD@NB=\hbox{\CD@FA
\CD@lA\CD@VA\CD@mA\CD@PA\z@\relax\CD@sF\ifnum\CD@FA<\CD@LB\CD@tB\wd\CD@VA
\relax\CD@eI\advance\CD@FA1 \advance\CD@VA1
\repeat}\CD@V\CD@kI{\wd\CD@NB}\wd
\CD@NB\z@}\def\CD@eI{\ifhbox\CD@FA\CD@OA\CD@tB\relax\advance\CD@OA-\CD@PA
\relax\ifdim\CD@OA=\z@\else\kern\CD@OA\fi\CD@PA\CD@tB\advance\CD@PA\wd\CD@FA
\relax\unhbox\CD@FA\advance\CD@PA-\lastkern\unkern\fi}\def\CD@ZI{\setbox
\CD@sH=\box\voidb@x\CD@VA=\CD@MB\CD@FA\CD@LB\CD@VA\CD@mA\advance\CD@VA\CD@FA
\advance\CD@VA-\CD@lA\advance\CD@VA\m@ne\CD@tB\wd\CD@VA\count@\CD@LB\advance
\count@\m@ne\CD@hF.5\wd\count@\advance\CD@hF\CD@tB\CD@A\m@ne\CD@gD\@m\CD@sF
\ifnum\CD@FA>\CD@lA\advance\CD@FA\m@ne\advance\CD@hF-\CD@tB\CD@PI\wd\CD@VA
\CD@tB\advance\CD@hF\CD@tB\advance\CD@VA\m@ne\CD@tB\wd\CD@VA\repeat\CD@mF
\CD@kF\CD@mI-\CD@mF\CD@vB}\newcount\CD@GB\def\CD@s{}\def\CD@t{\mathsurround
\z@\hsize\z@\rightskip\z@ plus1fil
minus\maxdimen\parfillskip\z@\linepenalty 9000 \looseness0
\hfuzz\maxdimen\hbadness10000 \clubpenalty0 \widowpenalty0
\displaywidowpenalty0 \interlinepenalty0 \predisplaypenalty0
\postdisplaypenalty0 \interdisplaylinepenalty0
\interfootnotelinepenalty0 \floatingpenalty0 \brokenpenalty0
\everypar{}\leftskip\z@\parskip\z@
\parindent\z@\pretolerance10000 \tolerance10000 \hyphenpenalty10000
\exhyphenpenalty10000 \binoppenalty10000 \relpenalty10000
\adjdemerits0 \doublehyphendemerits0 \finalhyphendemerits0
\baselineskip\z@\CD@IA\prevdepth
\z@}\newbox\CD@KG\newbox\CD@IG\def\CD@JG{\unhcopy\CD@KG}\def\CD@HG{\unhcopy
\CD@IG}\def\CD@iJ{\hbox{}\penalty1\nointerlineskip}\def\CD@PI{\penalty5
\noindent\setbox\CD@MH=\null\CD@mF\z@\CD@mI\z@\ifnum\CD@FA<\CD@LB\ht\CD@MH\ht
\CD@FA\dp\CD@MH\dp\CD@FA\unhbox\CD@FA\skip0=\lastskip\unskip\else\CD@OK\skip0%
=\z@\fi\endgraf\ifcase\prevgraf\cd@shouldnt Y \or\cd@shouldnt Z
\or\CD@RI\or
\CD@XI\else\CD@QI\fi\unskip\setbox0=\lastbox\unskip\unskip\unpenalty\noindent
\unhbox0\setbox0\lastbox\unpenalty\unskip\unskip\unpenalty\setbox0\lastbox
\CD@tF\CD@GB\lastpenalty\unpenalty\ifnum\CD@GB>\z@\setbox\z@\lastbox\CD@lB
\repeat\endgraf\unskip\unskip\unpenalty}\def\CD@YJ{\CD@uA\CD@XB\advance\CD@uA
-\CD@NB\CD@vA\CD@FA\advance\CD@vA-\CD@lA\advance\CD@vA1 \expandafter\message{%
prevgraf=\the\prevgraf at
(\the\CD@uA,\the\CD@vA)}}\def\CD@XI{\CD@CE\setbox
\CD@lI=\lastbox\setbox\CD@lI=\hbox{\unhbox\CD@lI\unskip\unpenalty}\unskip
\ifdim\ht\CD@lI>\ht\CD@PC\setbox\CD@MH=\copy\CD@lI\else\ifdim\dp\CD@lI>\dp
\CD@PC\setbox\CD@MH=\copy\CD@lI\else\CD@FG\CD@lI\fi\fi\advance\CD@mF.5\wd
\CD@lI\advance\CD@mI.5\wd\CD@lI\setbox\CD@lI=\hbox{\unhbox\CD@lI\CD@HG}\CD@bH
\CD@mF{\box\CD@lI}\CD@mI\z@\CD@yB\CD@vB}\def\CD@CE{\ifnum\CD@A>0
\advance \dimen0-\CD@tB\CD@iA-.5\dimen0 \CD@A-\CD@A\else\CD@A0
\CD@iA\z@\fi\setbox
\CD@MH=\lastbox\setbox\CD@MH=\hbox{\unhbox\CD@MH\unskip\unskip\unpenalty
\setbox0=\lastbox\global\CD@QA\lastkern\unkern}\advance\CD@iA-.5\CD@QA\unskip
\setbox\CD@MH=\null\CD@mI\CD@iA\CD@mF-\CD@iA}\def\CD@Z{\ht\CD@MH\CD@tI\dp
\CD@MH\CD@sI}\def\CD@FG#1{\setbox\CD@MH=\hbox{\CD@V{\ht\CD@MH}{\ht#1}\CD@V{%
\dp\CD@MH}{\dp#1}\CD@V{\wd\CD@MH}{\wd#1}\vrule height\ht\CD@MH
depth\dp\CD@MH
width\wd\CD@MH}}\def\CD@QI{\CD@CE\CD@Z\setbox\CD@lI=\lastbox\unskip\setbox
\CD@lF=\lastbox\unskip\setbox\CD@lF=\hbox{\unhbox\CD@lF\unskip\global\CD@yA
\lastpenalty\unpenalty}\advance\CD@yA9999
\ifcase\CD@yA\CD@VI\or\CD@YI\or
\CD@TI\or\CD@dI\or\CD@cI\or\CD@SI\else\cd@shouldnt9\fi}\def\CD@VI{\CD@FG
\CD@lI\CD@UI\setbox\CD@sH=\box\CD@lF\setbox\CD@tH=\box\CD@lI}\def\CD@YI{%
\CD@FG\CD@lF\setbox\CD@lI\hbox{\penalty8
\unhbox\CD@lI\unskip\unpenalty\ifnum \lastpenalty=8
\else\CD@xH\fi}\CD@UI\setbox\CD@lF=\hbox{\unhbox\CD@lF\unskip
\unpenalty\global\setbox\CD@DA=\lastbox}\ifdim\wd\CD@lF=\z@\else\CD@xH\fi
\setbox\CD@sH=\box\CD@DA}\def\CD@xH{\CD@KB{extra material in
\string\pile \space cell
(lost)}}\def\CD@UI{\CD@yB\ifvoid\CD@sH\else\CD@KB{Clashing
horizontal
arrows}\CD@mI.5\CD@hF\CD@mF-\CD@mI\CD@vB\CD@mI\z@\CD@mF\z@\fi
\CD@hI\CD@hF\advance\CD@hI-\CD@mI\CD@hF-\CD@mF\CD@JC\CD@FA}\def\CD@RI{\setbox
0\lastbox\unskip\CD@iA\z@\CD@Z\ifdim\skip0>\z@\CD@tJ\CD@A0
\else\ifnum\CD@A<1
\CD@A0 \dimen0\CD@tB\fi\advance\CD@A1 \fi}\def\VonH{\CD@MA46\VonH{.5\CD@LF}}%
\def\HonV{\CD@MA57\HonV{.5\CD@LF}}\def\HmeetV{\CD@MA44\HmeetV{-\MapShortFall}%
}\def\CD@MA#1#2#3#4{\CD@pB34#1{\string#3}\CD@SD\CD@GB-999#2
\dimen0=#4\CD@tI
\dimen0\advance\CD@tI\axisheight\CD@sI\dimen0\advance\CD@sI-\axisheight\CD@CF
\CD@HC\CD@ZD}\def\CD@HC#1{\setbox0=\hbox{\CD@k#1\CD@ND}\dimen0.5\wd0
\CD@tI
\ht0 \CD@sI\dp0 \CD@ZD}\def\CD@SD{\setbox0=\null\ht0=\CD@tI\dp0=\CD@sI\wd0=%
\dimen0 \copy0\penalty\CD@GB\box0
}\def\CD@TI{\CD@GC\CD@yB}\def\CD@dI{\CD@GC
\CD@vB}\def\CD@SI{\CD@GC\CD@yB\CD@vB}\def\CD@GC{\setbox\CD@lI=\hbox{\unhbox
\CD@lI}\setbox\CD@lF=\hbox{\unhbox\CD@lF\global\setbox\CD@DA=\lastbox}\ht
\CD@MH\ht\CD@DA\dp\CD@MH\dp\CD@DA\advance\CD@mF\wd\CD@DA\advance\CD@mI\wd
\CD@lI}\CD@tG\ifPositiveGradient\CD@CI\CD@BI\CD@CI\CD@tG\ifClimbing\CD@rB
\CD@qB\CD@rB\newcount\DiagonalChoice\DiagonalChoice\m@ne\ifx\tenln\nullfont
\CD@tJ\def\CD@qF{\CD@KH\ifPositiveGradient/\else\CD@k\backslash\CD@ND\fi}%
\else\def\CD@qF{\CD@rF\char\count@}\fi\let\CD@rF\tenln\def\Use@line@char#1{%
\hbox{#1\CD@rF\ifPositiveGradient\else\advance\count@64
\fi\char\count@}}\def
\CD@cF{\Use@line@char{\count@\CD@TC\multiply\count@8\advance\count@-9\advance
\count@\CD@LH}}\def\CD@ZF{\Use@line@char{\ifcase\DiagonalChoice\CD@gF\or
\CD@fF\or\CD@fF\else\CD@gF\fi}}\def\CD@gF{\ifnum\CD@TC=\z@\count@'33
\else
\count@\CD@TC\multiply\count@\sixt@@n\advance\count@-9\advance\count@\CD@LH
\advance\count@\CD@LH\fi}\def\CD@fF{\count@'\ifcase\CD@LH55\or\ifcase\CD@TC66%
\or22\or52\or61\or72\fi\or\ifcase\CD@TC66\or25\or22\or63\or52\fi\or\ifcase
\CD@TC66\or16\or36\or22\or76\fi\or\ifcase\CD@TC66\or27\or25\or67\or22\fi\fi
\relax}\def\CD@uC#1{\hbox{#1\setbox0=\Use@line@char{#1}\ifPositiveGradient
\else\raise.3\ht0\fi\copy0 \kern-.7\wd0
\ifPositiveGradient\raise.3\ht0\fi
\box0}}\def\CD@jF#1{\hbox{\setbox0=#1\kern-.75\wd0 \vbox to.25\ht0{%
\ifPositiveGradient\else\vss\fi\box0 \ifPositiveGradient\vss\fi}}}\def\CD@jI#%
1{\hbox{\setbox0=#1\dimen0=\wd0 \vbox
to.25\ht0{\ifPositiveGradient\vss\fi
\box0 \ifPositiveGradient\else\vss\fi}\kern-.75\dimen0 }}\CD@RC{+h:>}{%
\Use@line@char\CD@fF}\CD@RC{-h:>}{\Use@line@char\CD@gF}\CD@nF{+t:<}{-h:>}%
\CD@nF{-t:<}{+h:>}\CD@RC{+t:>}{\CD@jF{\Use@line@char\CD@fF}}\CD@RC{-t:>}{%
\CD@jI{\Use@line@char\CD@gF}}\CD@nF{+h:<}{-t:>}\CD@nF{-h:<}{+t:>}\CD@RC{+h:>>%
}{\CD@uC\CD@fF}\CD@RC{-h:>>}{\CD@uC\CD@gF}\CD@nF{+t:<<}{-h:>>}\CD@nF{-t:<<}{+%
h:>>}\CD@nF{+h:>->}{+h:>>}\CD@nF{-h:>->}{-h:>>}\CD@nF{+t:<-<}{-h:>>}\CD@nF{-t%
:<-<}{+h:>>}\CD@RC{+t:>>}{\CD@jF{\CD@uC\CD@fF}}\CD@RC{-t:>>}{\CD@jI{\CD@uC
\CD@gF}}\CD@nF{+h:<<}{-t:>>}\CD@nF{-h:<<}{+t:>>}\CD@nF{+t:>->}{+t:>>}\CD@nF{-%
t:>->}{-t:>>}\CD@nF{+h:<-<}{-t:>>}\CD@nF{-h:<-<}{+t:>>}\CD@RC{+f:-}{\CD@EF
\null\else\CD@cF\fi}\CD@nF{-f:-}{+f:-}\def\CD@tC#1#2{\vbox to#1{\vss\hbox to#%
2{\hss.\hss}\vss}}\def\hfdot{\CD@tC{2\axisheight}{.5em}}%
\def\vfdot{\CD@tC{1ex}\z@}
\def\CD@bF{\hbox{\dimen0=.3\CD@zC\dimen1\dimen0 \ifnum\CD@LH>\CD@TC\CD@iC{%
\dimen1}\else\CD@dG{\dimen0}\fi\CD@tC{\dimen0}{\dimen1}}}\newarrowfiller{.}%
\hfdot\hfdot\vfdot\vfdot\def\dfdot{\CD@bF\CD@CK}\CD@RC{+f:.}{\dfdot}\CD@RC{-f%
:.}{\dfdot}\def\CD@@K#1{\hbox\bgroup\def\CD@CH{#1\egroup}\afterassignment
\CD@CH
\count@='}\def\lnchar{\CD@@K\CD@qF}\def\CD@dF#1{\setbox#1=\hbox{\dimen5\dimen
#1 \setbox8=\box#1 \dimen1\wd8 \count@\dimen5 \divide\count@\dimen1
\ifnum \count@=0 \box8 \ifdim\dimen5<.95\dimen1 \CD@gB{diagonal line
too short}\fi
\else\dimen3=\dimen5 \advance\dimen3-\dimen1 \divide\dimen3\count@\dimen4%
\dimen3 \CD@dG{\dimen4}\ifPositiveGradient\multiply\dimen4\m@ne\fi\dimen6%
\dimen1 \advance\dimen6-\dimen3 \loop\raise\count@\dimen4\copy8
\ifnum\count@
>0 \kern-\dimen6 \advance\count@\m@ne\repeat\fi}}\def\CD@CG#1{\CD@EF\CD@xJ{#1%
}\else\CD@dF{#1}\fi}\def\CD@IH#1{}\newdimen\objectheight\objectheight1.8ex
\newdimen\objectwidth\objectwidth1em \def\CD@YD{\dimen6=\CD@aK
\DiagramCellHeight\dimen7=\CD@WK\DiagramCellWidth\CD@KJ\ifnum\CD@LH>0
\ifnum
\CD@TC>0 \CD@aF\else\aftergroup\CD@VC\fi\else\aftergroup\CD@UC\fi}\def\CD@VC{%
\CD@YA{diagonal map is nearly
vertical}\CD@NA}\def\CD@UC{\CD@YA{diagonal map is nearly
horizontal}\CD@NA}\CD@rG\CD@NA{Use an orthogonal map instead}\def
\CD@aF{\CD@MJ\dimen3\dimen7\dimen7\dimen6\CD@iC{\dimen7}\advance\dimen3-%
\dimen7 \CD@MF\ifnum\CD@LH>\CD@TC\advance\dimen6-\dimen1\advance\dimen6-%
\dimen5 \CD@iC{\dimen1}\CD@iC{\dimen5}\else\dimen0\dimen1\advance\dimen0%
\dimen5\CD@dG{\dimen0}\advance\dimen6-\dimen0
\fi\dimen2.5\dimen7\advance \dimen2-\dimen1
\dimen4.5\dimen7\advance\dimen4-\dimen5 \ifPositiveGradient
\dimen0\dimen5
\advance\dimen1-\CD@WK\DiagramCellWidth\advance\dimen1 \CD@ZK
\DiagramCellWidth\setbox6=\llap{\unhbox6\kern.1\ht2}\setbox7=\rlap{\kern.1\ht
2\unhbox7}\else\dimen0\dimen1 \advance\dimen1-\CD@ZK\DiagramCellWidth\setbox7%
=\llap{\unhbox7\kern.1\ht2}\setbox6=\rlap{\kern.1\ht2\unhbox6}\fi\setbox6=%
\vbox{\box6\kern.1\wd2}\setbox7=\vtop{\kern.1\wd2\box7}\CD@dG{\dimen0}%
\advance\dimen0-\axisheight\advance\dimen0-\CD@bK\DiagramCellHeight\dimen5-%
\dimen0 \advance\dimen0\dimen6 \advance\dimen1.5\dimen3
\ifdim\wd3>\z@\ifdim \ht3>-\dp3\CD@TB\fi\fi\dimen3\dimen2
\dimen7\dimen2\advance\dimen7\dimen4
\ifvoid3 \else\CD@tE\else\dimen8\ht3\advance\dimen8-\axisheight\CD@iC{\dimen8%
}\CD@X{\dimen8}{.5\wd3}\dimen9\dp3\advance\dimen9\axisheight\CD@iC{\dimen9}%
\CD@X{\dimen9}{.5\wd3}\ifPositiveGradient\advance\dimen2-\dimen9\advance
\dimen4-\dimen8 \else\advance\dimen4-\dimen9\advance\dimen2-\dimen8
\fi\fi
\advance\dimen3-.5\wd3 \fi\dimen9=\CD@aK\DiagramCellHeight\advance\dimen9-2%
\DiagramCellHeight\CD@tE\advance\dimen2\dimen4 \CD@CG{2}\dimen2-\dimen0%
\advance\dimen2\dp2 \else\CD@CG{2}\CD@CG{4}\ifPositiveGradient\dimen2-\dimen0%
\advance\dimen2\dp2 \dimen4\dimen5\advance\dimen4-\ht4 \else\dimen4-\dimen0%
\advance\dimen4\dp4 \dimen2\dimen5\advance\dimen2-\ht2 \fi\fi\setbox0=\hbox to%
\z@{\kern\dimen1 \ifvoid1
\else\ifPositiveGradient\advance\dimen0-\dp1 \lower \dimen0
\else\advance\dimen5-\ht1 \raise\dimen5 \fi\rlap{\unhbox1}\fi\raise
\dimen2\rlap{\unhbox2}\ifvoid3 \else\lower.5\dimen9\rlap{\kern\dimen3\unhbox3%
}\fi\kern.5\dimen7 \lower.5\dimen9\box6 \lower.5\dimen9\box7
\kern.5\dimen7 \CD@tE\else\raise\dimen4\llap{\unhbox4}\fi\ifvoid5
\else\ifPositiveGradient \advance\dimen5-\ht5 \raise\dimen5
\else\advance\dimen0-\dp5 \lower\dimen0 \fi
\llap{\unhbox5}\fi\hss}\ht0=\axisheight\dp0=-\ht0\box0
}\def\NorthWest{\CD@BI \CD@rB\DiagonalChoice0
}\def\NorthEast{\CD@CI\CD@rB\DiagonalChoice1 }\def
\SouthWest{\CD@CI\CD@qB\DiagonalChoice3 }\def\SouthEast{\CD@BI\CD@qB
\DiagonalChoice2 }\def\CD@aD{\vadjust{\CD@uA\CD@FA\advance\CD@uA
\ifPositiveGradient\else-\fi\CD@ZK\relax\CD@vA\CD@NB\advance\CD@vA-\CD@bK
\relax\hbox{\advance\CD@uA\ifPositiveGradient-\fi\CD@WK\advance\CD@vA\CD@aK
\hbox{\box6 \kern\CD@DC\kern\CD@eJ\penalty1 \box7
\box\z@}\penalty\CD@uA
\penalty\CD@vA}\penalty\CD@uA\penalty\CD@vA\penalty104}}\def\CD@eH#1{\relax
\vadjust{\hbox@maths{#1}\penalty\CD@FA\penalty\CD@NB\penalty\tw@}}\def\CD@lB{%
\ifcase\CD@GB\or\or\CD@bH{.5\wd0}{\box0}{.5\wd0}\z@\or\unhbox\z@\setbox\z@
\lastbox\CD@bH{.5\wd0}{\box0}{.5\wd0}\z@\unpenalty\unpenalty\setbox\z@
\lastbox\or\CD@TG\else\advance\CD@GB-100
\ifnum\CD@GB<\z@\cd@shouldnt B\fi
\setbox\z@\hbox{\kern\CD@mF\copy\CD@MH\kern\CD@mI\CD@uA\CD@XB\advance\CD@uA-%
\CD@NB\penalty\CD@uA\CD@uA\CD@FA\advance\CD@uA-\CD@lA\penalty\CD@uA\unhbox\z@
\global\CD@yA\lastpenalty\unpenalty\global\CD@zA\lastpenalty\unpenalty}\CD@uA
-\CD@yA\CD@vA\CD@zA\CD@fI\fi}\def\CD@TG{\unhbox\z@\setbox\z@\lastbox\CD@uA
\lastpenalty\unpenalty\advance\CD@uA\CD@mA\CD@vA\CD@XB\advance\CD@vA-%
\lastpenalty\unpenalty\dimen1\lastkern\unkern\setbox3\lastbox\dimen0\lastkern
\unkern\setbox0=\hbox to\z@{\unhbox0\setbox0\lastbox\setbox7\lastbox
\unpenalty\CD@eJ\lastkern\unkern\CD@DC\lastkern\unkern\setbox6\lastbox\dimen7%
\CD@tB\advance\dimen7-\wd\CD@uA\ifdim\dimen7<\z@\CD@CI\multiply\dimen7\m@ne
\let\mv\empty\else\CD@BI\def\mv{\raise\ht1}\kern-\dimen7 \fi\ifnum\CD@vA>%
\CD@NB\dimen6\CD@uB\advance\dimen6-\ht\CD@vA\else\dimen6\z@\fi\CD@jJ\CD@mK
\setbox1\null\ht1\dimen6\wd1\dimen7 \dimen7\dimen2 \dimen6\wd1
\CD@KJ\CD@uA \CD@LH\CD@vA\CD@TC\dimen6\ht1
\CD@KJ\setbox2\null\divide\dimen2\tw@\advance
\dimen2\CD@eJ\CD@eG{\dimen2}\wd2\dimen2 \dimen0.5\dimen7 \advance\dimen0%
\ifPositiveGradient\else-\fi\CD@eJ\CD@dG{\dimen0}\advance\dimen0-\axisheight
\ht2\dimen0
\dimen0\CD@DC\CD@eG{\dimen0}\advance\dimen0\ht2\ht2\dimen0 \dimen
0\ifPositiveGradient-\fi\CD@DC\CD@dG{\dimen0}\advance\dimen0\wd2\wd2\dimen0
\setbox4\null\dimen0 .6\CD@zC\CD@eG{\dimen0}\ht4\dimen0 \dimen0
.2\CD@zC \CD@dG{\dimen0}\wd4\dimen0 \dimen0\wd2
\ifvoid6\else\dimen1\ht4 \advance
\dimen1\ht2 \CD@CC6+-\raise\dimen1\rlap{\ifPositiveGradient\advance\dimen0-%
\wd6\advance\dimen0-\wd4 \else\advance\dimen0\wd4
\fi\kern\dimen0\box6}\fi \dimen0\wd2 \ifvoid7\else\dimen1\ht4
\advance\dimen1-\ht2 \CD@CC7-+\lower
\dimen1\rlap{\ifPositiveGradient\advance\dimen0\wd4 \else\advance\dimen0-\wd7%
\advance\dimen0-\wd4
\fi\kern\dimen0\box7}\fi\mv\box0\hss}\ht0\z@\dp0\z@
\CD@bH{\z@}{\box\z@}{\z@}{\axisheight}}\def\CD@CC#1#2#3{\dimen4.5\wd#1
\ifdim
\dimen4>.25\dimen7\dimen4=.25\dimen7\fi\ifdim\dimen4>\CD@zC\dimen4.4\dimen4
\advance\dimen4.6\CD@zC\fi\CD@eG{\dimen4}\dimen5\axisheight\CD@dG{\dimen5}%
\advance\dimen4-\dimen5 \dimen5\dimen4\CD@eG{\dimen5}\advance\dimen0%
\ifPositiveGradient#2\else#3\fi\dimen5
\CD@dG{\dimen4}\advance\dimen1\dimen4 }
\def\CD@eD#1{\expandafter\CD@IK{#1}}\CD@ZA\CD@EK{output is PostScript
dependent}\def\CD@SC{\CD@IK{/bturn {gsave currentpoint currentpoint
translate
4 2 roll neg exch atan rotate neg exch neg exch translate } def /eturn {%
currentpoint grestore moveto}
def}}\def\CD@gK{\relax\CD@hK\CD@tK{Q}\else \CD@IK{eturn}\fi}
\def\CD@OJ#1{\count@#1\relax\multiply\count@7\advance
\count@16577\divide\count@33154
}\def\CD@fD#1{\expandafter\special{#1}} \def
\CD@xJ#1{\setbox#1=\hbox{\dimen0\dimen#1\CD@dG{\dimen0}\CD@OJ{\dimen0}\setbox
0=\null\ifPositiveGradient\count@-\count@\ht0\dimen0
\else\dp0\dimen0 \fi\box
0 \CD@uA\count@\CD@OJ\CD@LF\CD@fD{pn \the\count@}\CD@fD{pa 0 0}\CD@OJ{\dimen#%
1}\CD@fD{pa \the\count@\space\the\CD@uA}\CD@fD{fp}\kern\dimen#1}}\def\CD@JI{%
\CD@KJ\begingroup\ifdim\dimen7<\dimen6 \dimen2=\dimen6
\dimen6=\dimen7 \dimen 7=\dimen2
\count@\CD@LH\CD@LH\CD@TC\CD@TC\count@\else\dimen2=\dimen7 \fi
\ifdim\dimen6>.01\p@\CD@KI\global\CD@QA\dimen0
\else\global\CD@QA\dimen7 \fi
\endgroup\dimen2\CD@QA\CD@iK\CD@lK{\ifPositiveGradient\else-\fi\dimen6}\CD@iK
\CD@kK{\ifPositiveGradient-\fi\dimen6}\CD@iK\CD@eK{\dimen7}}\def\CD@KI{\CD@hJ
\ifdim\dimen7>1.73\dimen6 \divide\dimen2 4 \multiply\CD@TC2 \else\dimen2=0.%
353553\dimen2 \advance\CD@LH-\CD@TC\multiply\CD@TC4
\fi\dimen0=4\dimen2 \CD@ZG
4\CD@ZG{-2}\CD@ZG2\CD@ZG{-2.5}}\def\CD@AI{\begingroup\count@\dimen0 \dimen2 45%
pt \divide\count@\dimen2
\ifdim\dimen0<\z@\advance\count@\m@ne\fi\ifodd
\count@\advance\count@1\CD@@A\else\CD@y\fi\advance\dimen0-\count@\dimen2
\CD@gE\multiply\dimen0\m@ne\fi\ifnum\count@<0 \multiply\count@-7 \fi\dimen3%
\dimen1 \dimen6\dimen0 \dimen7 3754936sp \ifdim\dimen0<6\p@\def\CD@OG{4000}%
\fi\CD@KJ\dimen2\dimen3\CD@dG{\dimen2}\CD@hJ\multiply\CD@TC-6
\dimen0\dimen2
\CD@ZG1\CD@ZG{0.3}\dimen1\dimen0 \dimen2\dimen3 \dimen0\dimen3 \CD@ZG3\CD@ZG{%
1.5}\CD@ZG{0.3}\divide\count@2
\CD@gE\multiply\dimen1\m@ne\fi\ifodd\count@
\dimen2\dimen1\dimen1\dimen0\dimen0-\dimen2 \fi\divide\count@2
\ifodd\count@
\multiply\dimen0\m@ne\multiply\dimen1\m@ne\fi\global\CD@QA\dimen0\global
\CD@RA\dimen1\endgroup\dimen6\CD@QA\dimen7\CD@RA}\def\CD@OC{255}\let\CD@OG
\CD@OC\def\CD@KJ{\begingroup\ifdim\dimen7<\dimen6
\dimen9\dimen7\dimen7\dimen
6\dimen6\dimen9\CD@@A\else\CD@y\fi\dimen2\z@\dimen3\CD@XH\dimen4\CD@XH\dimen0%
\z@\dimen8=\CD@OG\CD@XH\CD@lC\global\CD@yA\dimen\CD@gE0\else3\fi\global\CD@zA
\dimen\CD@gE3\else0\fi\endgroup\CD@LH\CD@yA\CD@TC\CD@zA}\def\CD@lC{\count@
\dimen6 \divide\count@\dimen7 \advance\dimen6-\count@\dimen7
\dimen9\dimen4 \advance\dimen9\count@\dimen0 \ifdim\dimen9>\dimen8
\CD@@C\else\CD@AC\ifdim \dimen6>\z@\dimen9\dimen6 \dimen6\dimen7
\dimen7\dimen9 \expandafter
\expandafter\expandafter\CD@lC\fi\fi}\def\CD@@C{\ifdim\dimen0=\z@\ifdim\dimen
9<2\dimen8 \dimen0\dimen8 \fi\else\advance\dimen8-\dimen4 \divide\dimen8%
\dimen0 \ifdim\count@\CD@XH<2\dimen8 \count@\dimen8 \dimen9\dimen4
\advance \dimen9\count@\dimen0
\CD@AC\fi\fi}\def\CD@AC{\dimen4\dimen0 \dimen0\dimen9
\advance\dimen2\count@\dimen3 \dimen9\dimen2 \dimen2\dimen3 \dimen3\dimen9 }%
\def\CD@ZG#1{\CD@dG{\dimen2}\advance\dimen0 #1\dimen2 }\def\CD@dG#1{\divide#1%
\CD@TC\multiply#1\CD@LH}\def\CD@eG#1{\divide#1\CD@vA\multiply#1\CD@uA}\def
\CD@iC#1{\divide#1\CD@LH\multiply#1\CD@TC}\def\CD@hJ{\dimen6\CD@LH\CD@XH
\multiply\dimen6\CD@LH\dimen7\CD@TC\CD@XH\multiply\dimen7\CD@TC\CD@KJ}\def
\CD@iK#1#2{\begingroup\dimen@#2\relax\loop\ifdim\dimen2<.4\maxdimen\multiply
\dimen2\tw@\multiply\dimen@\tw@\repeat\divide\dimen2\@cclvi\divide\dimen@
\dimen2\relax\multiply\dimen@\@cclvi\expandafter\CD@jK\the\dimen@\endgroup
\let#1\CD@fK}{\catcode`p=12 \catcode`0=12 \catcode`.=12 \catcode`t=12 \gdef
\CD@jK#1pt{\gdef\CD@fK{#1}}}\ifx\errorcontextlines\CD@qK\CD@tJ\let\CD@GH
\relax\else\def\CD@GH{\errorcontextlines\m@ne}\fi\ifnum\inputlineno<0
\let
\CD@CD\empty\let\CD@W\empty\let\CD@mD\relax\let\CD@uI\relax\let\CD@vI\relax
\let\CD@zF\relax\else\def\CD@W{ at line \number\inputlineno}\def\CD@mD{ - first occurred%
}\def\CD@uI{\edef\CD@h{\the\inputlineno}\global\let\CD@jB\CD@h}\def\CD@h{9999%
}\def\CD@vI{\xdef\CD@jB{\the\inputlineno}}\def\CD@jB{\CD@h}\def\CD@zF{\ifnum
\CD@h<\inputlineno\edef\CD@CD{\space at lines
\CD@h--\the\inputlineno}\else
\edef\CD@CD{\CD@W}\fi}\fi\let\CD@CD\empty\def\CD@YA#1#2{\CD@GH\errhelp=#2%
\expandafter\errmessage{\CD@tA:
#1}}\def\CD@KB#1{\begingroup\expandafter \message{! \CD@tA:
#1\CD@CD}\ifnum\CD@XB>\CD@NB\ifnum\CD@CA>\CD@NB\else\ifnum
\CD@lA>\CD@FA\else\ifnum\CD@LB>\CD@FA\advance\CD@XB-\CD@NB\advance\CD@FA-%
\CD@lA\advance\CD@FA1\relax\expandafter\message{! (error detected at
row \the \CD@XB, column \the\CD@FA, but probably caused
elsewhere)}\fi\fi\fi\fi
\endgroup}\def\CD@gB#1{{\expandafter\message{\CD@tA\space Warning: #1\CD@W}}}%
\def\CD@CB#1#2{\CD@gB{#1 \string#2 is obsolete\CD@mD}}\def\CD@AB#1{\CD@CB{%
Dimension}{#1}\CD@DE#1\CD@BB\CD@BB}\def\CD@BB{\CD@OA=}\def\CD@@B#1{\CD@CB{%
Count}{#1}\CD@DE#1\CD@OH\CD@OH}\def\CD@OH{\count@=}\def\HorizontalMapLength{%
\CD@AB\HorizontalMapLength}\def\VerticalMapHeight{\CD@AB\VerticalMapHeight}%
\def\VerticalMapDepth{\CD@AB\VerticalMapDepth}\def\VerticalMapExtraHeight{%
\CD@AB\VerticalMapExtraHeight}\def\VerticalMapExtraDepth{\CD@AB
\VerticalMapExtraDepth}\def\DiagonalLineSegments{\CD@@B\DiagonalLineSegments}%
\ifx\tenln\nullfont\CD@ZA\CD@KH{\CD@eF\space diagonal line and arrow
font not
available}\else\let\CD@KH\relax\fi\def\CD@aG#1#2<#3:#4:#5#6{\begingroup\CD@PA
#3\relax\advance\CD@PA-#2\relax\ifdim.1em<\CD@PA\CD@uA#5\relax\CD@vA#6\relax
\ifnum\CD@uA<\CD@vA\count@\CD@vA\advance\count@-\CD@uA\CD@KB{#4 by
\the\CD@PA
}\if#1v\let\CD@CH\CD@JK\edef\tmp{\the\CD@uA--\the\CD@vA,\the\CD@FA}\else
\advance\count@\count@\if#1l\advance\count@-\CD@A\else\if#1r\advance\count@
\CD@A\fi\fi\advance\CD@PA\CD@PA\let\CD@CH\CD@ZE\edef\tmp{\the\CD@NB,\the
\CD@uA--\the\CD@vA}\fi\divide\CD@PA\count@\ifdim\CD@CH<\CD@PA\global\CD@CH
\CD@PA\fi\fi\fi\endgroup}\CD@tG\CD@xE\CD@JD\CD@ID\CD@rG\CD@xI{See
the message above.}\CD@rG\CD@lH{Perhaps you've forgotten to end the
diagram before resuming the text, in\CD@uG which case some garbage
may be added to the diagram, but we should be ok now.\CD@uG
Alternatively you've left a blank line in the middle - TeX will now
complain\CD@uG that the remaining \CD@S s are misplaced - so please
use comments for layout.}\CD@rG\CD@hD{You have already
closed too many brace pairs or environments; an \CD@HD\CD@uG command was (%
over)due.}\CD@rG\CD@hH{\CD@dC\space and \CD@HD\space commands must match.}%
\def\CD@jH{\ifnum\inputlineno=0 \else\expandafter\CD@iH\fi}\def\CD@iH{\CD@MD
\CD@GD\crcr\CD@YA{missing \CD@HD\space inserted before \CD@kH- type "h"}%
\CD@lH\enddiagram\CD@AG\CD@kH\par}\def\CD@AG#1{\edef\enddiagram{\noexpand
\CD@rD{#1\CD@W}}}\def\CD@rD#1{\CD@YA{\CD@HD\space(anticipated by #1) ignored}%
\CD@xI\let\enddiagram\CD@SG}\def\CD@SG{\CD@YA{misplaced \CD@HD\space ignored}%
\CD@hH}\def\CD@mC{\CD@YA{missing \CD@HD\space
inserted.}\CD@hD\CD@AG{closing
group}}\ifx\ifx
\fi\fi

\catcode`\$=3 
\def\vboxtoz{\vbox to\z@}

\def\scriptaxis#1{\@scriptaxis{$\scriptstyle#1$}}
\def\ssaxis#1{\@ssaxis{$\scriptscriptstyle#1$}}
\def\@scriptaxis#1{\dimen0\axisheight\advance\dimen0-\ss@axisheight\raise
\dimen0\hbox{#1}}\def\@ssaxis#1{\dimen0\axisheight\advance\dimen0-%
\ss@axisheight\raise\dimen0\hbox{#1}}

\ifx\boldmath\CD@qK
\let\boldscriptaxis\scriptaxis
\def\boldscript#1{\hbox{$\scriptstyle#1$}}
\else\def\boldscriptaxis#1{\@scriptaxis{\boldmath$\scriptstyle#1$}}
\def\boldscript#1{\hbox{\boldmath$\scriptstyle#1$}}
\fi

\def\raisehook#1#2#3{\hbox{\setbox3=\hbox{#1$\scriptscriptstyle#3$}%
\dimen0\ss@axisheight
\dimen1\axisheight\advance\dimen1-\dimen0
\dimen2\ht3\advance\dimen2-\dimen0%
\advance\dimen2-0.021em\advance\dimen1 #2\dimen2%
\raise\dimen1\box3}}
\def\shifthook#1#2#3{\setbox1=\hbox{#1$\scriptscriptstyle#3$}\dimen0\wd1%
\divide\dimen0 12\CD@zH{\dimen0}
\dimen1\wd1\advance\dimen1-2\dimen0 \advance\dimen1-2\CD@oI\CD@zH{\dimen1}%
\kern#2\dimen1\box1}

\def\@cmex{\mathchar"03}

\def\make@pbk#1{\setbox\tw@\hbox to\z@{#1}\ht\tw@\z@\dp\tw@\z@\box\tw@}\def
\CD@fH#1{\overprint{\hbox to\z@{#1}}}\def\CD@qH{\kern0.11em}\def\CD@pH{\kern0%
.35em}

\def\dblvert{\def\CD@rH{\kern.5\PileSpacing}}\def\CD@rH{}

\def\SEpbk{\make@pbk{\CD@qH\CD@rH\vrule depth 2.87ex height -2.75ex width 0.%
95em \vrule height -0.66ex depth 2.87ex width 0.05em \hss}}

\def\SWpbk{\make@pbk{\hss\vrule height -0.66ex depth 2.87ex width 0.05em
\vrule depth 2.87ex height -2.75ex width 0.95em \CD@qH\CD@rH}}

\def\NEpbk{\make@pbk{\CD@qH\CD@rH\vrule depth -3.81ex height 4.00ex width 0.%
95em \vrule height 4.00ex depth -1.72ex width 0.05em \hss}}

\def\NWpbk{\make@pbk{\hss\vrule height 4.00ex depth -1.72ex width 0.05em
\vrule depth -3.81ex height 4.00ex width 0.95em \CD@qH\CD@rH}}

\def\puncture{{\setbox0\hbox{A}\vrule height.53\ht0 depth-.47\ht0 width.35\ht
0 \kern.12\ht0 \vrule height\ht0 depth-.65\ht0 width.06\ht0
\kern-.06\ht0 \vrule height.35\ht0 depth0pt width.06\ht0
\kern.12\ht0 \vrule height.53\ht0 depth-.47\ht0 width.35\ht0 }}

\def\NEclck{\overprint{\raise2.5ex\rlap{ \CD@rH$\scriptstyle\searrow$}}}
\def\NEanti{\overprint{\raise2.5ex\rlap{ \CD@rH$\scriptstyle\nwarrow$}}}
\def\NWclck{\overprint{\raise2.5ex\llap{$\scriptstyle\nearrow$ \CD@rH}}}
\def\NWanti{\overprint{\raise2.5ex\llap{$\scriptstyle\swarrow$ \CD@rH}}}
\def\SEclck{\overprint{\lower1ex\rlap{ \CD@rH$\scriptstyle\swarrow$}}}
\def\SEanti{\overprint{\lower1ex\rlap{ \CD@rH$\scriptstyle\nearrow$}}}
\def\SWclck{\overprint{\lower1ex\llap{$\scriptstyle\nwarrow$ \CD@rH}}}
\def\SWanti{\overprint{\lower1ex\llap{$\scriptstyle\searrow$ \CD@rH}}}

\def\rhvee{\mkern-10mu\greaterthan}
\def\lhvee{\lessthan\mkern-10mu}
\def\dhvee{\vboxtoz{\vss\hbox{$\vee$}\kern0pt}}
\def\uhvee{\vboxtoz{\hbox{$\wedge$}\vss}}
\newarrowhead{vee}\rhvee\lhvee\dhvee\uhvee

\def\dhlvee{\vboxtoz{\vss\hbox{$\scriptstyle\vee$}\kern0pt}}
\def\uhlvee{\vboxtoz{\hbox{$\scriptstyle\wedge$}\vss}}
\newarrowhead{littlevee}{\mkern1mu\scriptaxis\rhvee}{\scriptaxis\lhvee}%
\dhlvee\uhlvee\ifx\boldmath\CD@qK
\newarrowhead{boldlittlevee}{\mkern1mu\scriptaxis\rhvee}{\scriptaxis\lhvee}%
\dhlvee\uhlvee\else
\def\dhblvee{\vboxtoz{\vss\boldscript\vee\kern0pt}}
\def\uhblvee{\vboxtoz{\boldscript\wedge\vss}}
\newarrowhead{boldlittlevee}{\mkern1mu\boldscriptaxis\rhvee}{\boldscriptaxis
\lhvee}\dhblvee\uhblvee
\fi

\def\rhcvee{\mkern-10mu\succ}
\def\lhcvee{\prec\mkern-10mu}
\def\dhcvee{\vboxtoz{\vss\hbox{$\curlyvee$}\kern0pt}}
\def\uhcvee{\vboxtoz{\hbox{$\curlywedge$}\vss}}
\newarrowhead{curlyvee}\rhcvee\lhcvee\dhcvee\uhcvee

\def\rhvvee{\mkern-13mu\gg}
\def\lhvvee{\ll\mkern-13mu}
\def\dhvvee{\vboxtoz{\vss\hbox{$\vee$}\kern-.6ex\hbox{$\vee$}\kern0pt}}
\def\uhvvee{\vboxtoz{\hbox{$\wedge$}\kern-.6ex \hbox{$\wedge$}\vss}}
\newarrowhead{doublevee}\rhvvee\lhvvee\dhvvee\uhvvee

\def\rhtriangle{\triangleright\mkern1.2mu}
\def\lhtriangle{\triangleleft\mkern.8mu}
\def\uhtriangle{\vbox{\kern-.2ex \hbox{$\scriptscriptstyle\bigtriangleup$}%
\kern-.25ex}}
\def\dhtriangle{\vbox{\kern-.28ex \hbox{$\scriptscriptstyle\bigtriangledown$}%
\kern-.1ex}}
\def\dhblack{\vbox{\kern-.25ex\nointerlineskip\hbox{$\blacktriangledown$}}}%
\def\uhblack{\vbox{\kern-.25ex\nointerlineskip\hbox{$\blacktriangle$}}}%
\def\dhlblack{\vbox{\kern-.25ex\nointerlineskip\hbox{$\scriptstyle
\blacktriangledown$}}}
\def\uhlblack{\vbox{\kern-.25ex\nointerlineskip\hbox{$\scriptstyle
\blacktriangle$}}}
\newarrowhead{triangle}\rhtriangle\lhtriangle\dhtriangle\uhtriangle
\newarrowhead{blacktriangle}{\mkern-1mu\blacktriangleright\mkern.4mu}{%
\blacktriangleleft}\dhblack\uhblack\newarrowhead{littleblack}{\mkern-1mu%
\scriptaxis\blacktriangleright}{\scriptaxis\blacktriangleleft\mkern-2mu}%
\dhlblack\uhlblack

\def\rhla{\hbox{\setbox0=\lnchar55\dimen0=\wd0\kern-.6\dimen0\ht0\z@\raise
\axisheight\box0\kern.1\dimen0}}
\def\lhla{\hbox{\setbox0=\lnchar33\dimen0=\wd0\kern.05\dimen0\ht0\z@\raise
\axisheight\box0\kern-.5\dimen0}}
\def\dhla{\vboxtoz{\vss\rlap{\lnchar77}}}
\def\uhla{\vboxtoz{\setbox0=\lnchar66 \wd0\z@\kern-.15\ht0\box0\vss}}
\newarrowhead{LaTeX}\rhla\lhla\dhla\uhla

\def\lhlala{\lhla\kern.3em\lhla}
\def\rhlala{\rhla\kern.3em\rhla}
\def\uhlala{\hbox{\uhla\raise-.6ex\uhla}}
\def\dhlala{\hbox{\dhla\lower-.6ex\dhla}}
\newarrowhead{doubleLaTeX}\rhlala\lhlala\dhlala\uhlala

\def\hhO{\scriptaxis\bigcirc\mkern.4mu} \def\hho{{\circ}\mkern1.2mu}%
\newarrowhead{o}\hho\hho\circ\circ
\newarrowhead{O}\hhO\hhO{\scriptstyle\bigcirc}{\scriptstyle\bigcirc}

\def\rhtimes{\mkern-5mu{\times}\mkern-.8mu}\def\lhtimes{\mkern-.8mu{\times}%
\mkern-5mu}\def\uhtimes{\setbox0=\hbox{$\times$}\ht0\axisheight\dp0-\ht0%
\lower\ht0\box0 }\def\dhtimes{\setbox0=\hbox{$\times$}\ht0\axisheight\box0 }%
\newarrowhead{X}\rhtimes\lhtimes\dhtimes\uhtimes\newarrowhead+++++

\newarrowhead{Y}{\mkern-3mu\Yright}{\Yleft\mkern-3mu}\Ydown\Yup

\newarrowhead{->}\rightarrow\leftarrow\downarrow\uparrow

\newarrowhead{=>}\Rightarrow\Leftarrow{\@cmex7F}{\@cmex7E}

\newarrowhead{harpoon}\rightharpoonup\leftharpoonup\downharpoonleft
\upharpoonleft

\def\twoheaddownarrow{\rlap{$\downarrow$}\raise-.5ex\hbox{$\downarrow$}}
\def\twoheaduparrow{\rlap{$\uparrow$}\raise.5ex\hbox{$\uparrow$}}
\newarrowhead{->>}\twoheadrightarrow\twoheadleftarrow\twoheaddownarrow
\twoheaduparrow

\def\ltvee{\mkern-1mu{\lessthan}\mkern.4mu}
\newarrowtail{vee}\greaterthan\ltvee\vee\wedge

\newarrowtail{littlevee}{\scriptaxis\greaterthan}{\mkern-1mu\scriptaxis
\lessthan}{\scriptstyle\vee}{\scriptstyle\wedge}\ifx\boldmath\CD@qK
\newarrowtail{boldlittlevee}{\scriptaxis\greaterthan}{\mkern-1mu\scriptaxis
\lessthan}{\scriptstyle\vee}{\scriptstyle\wedge}\else\newarrowtail{%
boldlittlevee}{\boldscriptaxis\greaterthan}{\mkern-1mu\boldscriptaxis
\lessthan}{\boldscript\vee}{\boldscript\wedge}\fi

\newarrowtail{curlyvee}\succ{\mkern-1mu{\prec}\mkern.4mu}\curlyvee\curlywedge

\def\rttriangle{\mkern1.2mu\triangleright}
\newarrowtail{triangle}\rttriangle\lhtriangle\dhtriangle\uhtriangle
\newarrowtail{blacktriangle}\blacktriangleright{\mkern-1mu\blacktriangleleft
\mkern.4mu}\dhblack\uhblack\newarrowtail{littleblack}{\scriptaxis
\blacktriangleright\mkern-2mu}{\mkern-1mu\scriptaxis\blacktriangleleft}%
\dhlblack\uhlblack

\def\rtla{\hbox{\setbox0=\lnchar55\dimen0=\wd0\kern-.5\dimen0\ht0\z@\raise
\axisheight\box0\kern-.2\dimen0}}
\def\ltla{\hbox{\setbox0=\lnchar33\dimen0=\wd0\kern-.15\dimen0\ht0\z@\raise
\axisheight\box0\kern-.5\dimen0}}
\def\dtla{\vbox{\setbox0=\rlap{\lnchar77}\dimen0=\ht0\kern-.7\dimen0\box0%
\kern-.1\dimen0}}
\def\utla{\vbox{\setbox0=\rlap{\lnchar66}\dimen0=\ht0\kern-.1\dimen0\box0%
\kern-.6\dimen0}}
\newarrowtail{LaTeX}\rtla\ltla\dtla\utla

\def\rtvvee{\gg\mkern-3mu}
\def\ltvvee{\mkern-3mu\ll}
\def\dtvvee{\vbox{\hbox{$\vee$}\kern-.6ex \hbox{$\vee$}\vss}}
\def\utvvee{\vbox{\vss\hbox{$\wedge$}\kern-.6ex \hbox{$\wedge$}\kern\z@}}
\newarrowtail{doublevee}\rtvvee\ltvvee\dtvvee\utvvee

\def\ltlala{\ltla\kern.3em\ltla}
\def\rtlala{\rtla\kern.3em\rtla}
\def\utlala{\hbox{\utla\raise-.6ex\utla}}
\def\dtlala{\hbox{\dtla\lower-.6ex\dtla}}
\newarrowtail{doubleLaTeX}\rtlala\ltlala\dtlala\utlala

\def\utbar{\vrule height 0.093ex depth0pt width 0.4em}
\let\dtbar\utbar
\def\rtbar{\mkern1.5mu\vrule height 1.1ex depth.06ex width .04em\mkern1.5mu}%
\let\ltbar\rtbar
\newarrowtail{mapsto}\rtbar\ltbar\dtbar\utbar
\newarrowtail{|}\rtbar\ltbar\dtbar\utbar

\def\rthooka{\raisehook{}+\subset\mkern-1mu}
\def\lthooka{\mkern-1mu\raisehook{}+\supset}
\def\rthookb{\raisehook{}-\subset\mkern-2mu}
\def\lthookb{\mkern-1mu\raisehook{}-\supset}

\def\dthooka{\shifthook{}+\cap}
\def\dthookb{\shifthook{}-\cap}
\def\uthooka{\shifthook{}+\cup}
\def\uthookb{\shifthook{}-\cup}

\newarrowtail{hooka}\rthooka\lthooka\dthooka\uthooka\newarrowtail{hookb}%
\rthookb\lthookb\dthookb\uthookb

\ifx\boldmath\CD@qK\newarrowtail{boldhooka}\rthooka\lthooka\dthooka\uthooka
\newarrowtail{boldhookb}\rthookb\lthookb\dthookb\uthookb\newarrowtail{%
boldhook}\rthooka\lthooka\dthookb\uthooka\else\def\rtbhooka{\raisehook
\boldmath+\subset\mkern-1mu}
\def\ltbhooka{\mkern-1mu\raisehook\boldmath+\supset}
\def\rtbhookb{\raisehook\boldmath-\subset\mkern-2mu}
\def\ltbhookb{\mkern-1mu\raisehook\boldmath-\supset}
\def\dtbhooka{\shifthook\boldmath+\cap}
\def\dtbhookb{\shifthook\boldmath-\cap}
\def\utbhooka{\shifthook\boldmath+\cup}
\def\utbhookb{\shifthook\boldmath-\cup}
\newarrowtail{boldhooka}\rtbhooka\ltbhooka\dtbhooka\utbhooka\newarrowtail{%
boldhookb}\rtbhookb\ltbhookb\dtbhookb\utbhookb\newarrowtail{boldhook}%
\rtbhooka\ltbhooka\dtbhooka\utbhooka\fi

\def\dtsqhooka{\shifthook{}+\sqcap}
\def\ltsqhooka{\mkern-1mu\raisehook{}+\sqsupset}
\def\rtsqhooka{\raisehook{}+\sqsubset\mkern-1mu}
\def\utsqhooka{\shifthook{}+\sqcup}
\newarrowtail{sqhook}\rtsqhooka\ltsqhooka\dtsqhooka\utsqhooka

\newarrowtail{hook}\rthooka\lthookb\dthooka\uthooka\newarrowtail{C}\rthooka
\lthookb\dthooka\uthooka

\newarrowtail{o}\hho\hho\circ\circ
\newarrowtail{O}\hhO\hhO{\scriptstyle\bigcirc}{\scriptstyle\bigcirc}

\newarrowtail{X}\lhtimes\rhtimes\uhtimes\dhtimes\newarrowtail+++++

\newarrowtail{Y}\Yright\Yleft\Ydown\Yup

\newarrowtail{harpoon}\leftharpoondown\rightharpoondown\upharpoonright
\downharpoonright

\newarrowtail{<=}\Leftarrow\Rightarrow{\@cmex7E}{\@cmex7F}

\newarrowfiller{=}=={\@cmex77}{\@cmex77}
\def\vfthree{\mid\!\!\!\mid\!\!\!\mid}
\newarrowfiller{3}\equiv\equiv\vfthree\vfthree

\def\vfdashstrut{\vrule width0pt height1.3ex depth0.7ex}
\def\vfthedash{\vrule width\CD@LF height0.6ex depth 0pt}
\def\hfthedash{\CD@AJ\vrule\horizhtdp width 0.26em}
\def\hfdash{\mkern5.5mu\hfthedash\mkern5.5mu}
\def\vfdash{\vfdashstrut\vfthedash}
\newarrowfiller{dash}\hfdash\hfdash\vfdash\vfdash

\newarrowmiddle+++++

\iffalse
\newarrow{To}----{vee}
\newarrow{Arr}----{LaTeX}
\newarrow{Dotsto}....{vee}
\newarrow{Dotsarr}....{LaTeX}
\newarrow{Dashto}{}{dash}{}{dash}{vee}
\newarrow{Dasharr}{}{dash}{}{dash}{LaTeX}
\newarrow{Mapsto}{mapsto}---{vee}
\newarrow{Mapsarr}{mapsto}---{LaTeX}
\newarrow{IntoA}{hooka}---{vee}
\newarrow{IntoB}{hookb}---{vee}
\newarrow{Embed}{vee}---{vee}
\newarrow{Emarr}{LaTeX}---{LaTeX}
\newarrow{Onto}----{doublevee}
\newarrow{Dotsonarr}....{doubleLaTeX}
\newarrow{Dotsonto}....{doublevee}
\newarrow{Dotsonarr}....{doubleLaTeX}
\else
\newarrow{To}---->
\newarrow{Arr}---->
\newarrow{Dotsto}....>
\newarrow{Dotsarr}....>
\newarrow{Dashto}{}{dash}{}{dash}>
\newarrow{Dasharr}{}{dash}{}{dash}>
\newarrow{Mapsto}{mapsto}--->
\newarrow{Mapsarr}{mapsto}--->
\newarrow{IntoA}{hooka}--->
\newarrow{IntoB}{hookb}--->
\newarrow{Embed}>--->
\newarrow{Emarr}>--->
\newarrow{Onto}----{>>}
\newarrow{Dotsonarr}....{>>}
\newarrow{Dotsonto}....{>>}
\newarrow{Dotsonarr}....{>>}
\fi

\newarrow{Implies}===={=>}
\newarrow{Project}----{triangle}
\newarrow{Pto}----{harpoon}
\newarrow{Relto}{harpoon}---{harpoon}

\newarrow{Eq}=====
\newarrow{Line}-----
\newarrow{Dots}.....
\newarrow{Dashes}{}{dash}{}{dash}{}

\newarrow{SquareInto}{sqhook}--->

\newarrowhead{cmexbra}{\@cmex7B}{\@cmex7C}{\@cmex3B}{\@cmex38}
\newarrowtail{cmexbra}{\@cmex7A}{\@cmex7D}{\@cmex39}{\@cmex3A}
\newarrowmiddle{cmexbra}{\braceru\bracelu}{\bracerd\braceld}{\vcenter{%
\hbox@maths{\@cmex3D\mkern-2mu}}}
{\vcenter{\hbox@maths{\mkern2mu\@cmex3C}}}
\newarrow{@brace}{cmexbra}-{cmexbra}-{cmexbra}
\newarrow{@parenth}{cmexbra}---{cmexbra}
\def\rightBrace{\d@brace[thick,cmex]}
\def\leftBrace{\u@brace[thick,cmex]}
\def\upperBrace{\r@brace[thick,cmex]}
\def\lowerBrace{\l@brace[thick,cmex]}
\def\rightParenth{\d@parenth[thick,cmex]}
\def\leftParenth{\u@parenth[thick,cmex]}
\def\upperParenth{\r@parenth[thick,cmex]}
\def\lowerParenth{\l@parenth[thick,cmex]}

\let\hEq\rEq
\let\vEq\uEq

\def\labelstyle{
\ifincommdiag
\textstyle
\else
\scriptstyle
\fi}
\let\objectstyle\displaystyle

\newdiagramgrid{pentagon}{0.618034,0.618034,1,1,1,1,0.618034,0.618034}{1.%
17557,1.17557,1.902113,1.902113}

\newdiagramgrid{perspective}{0.75,0.75,1.1,1.1,0.9,0.9,0.95,0.95,0.75,0.75}{0%
.75,0.75,1.1,1.1,0.9,0.9}

\diagramstyle[
dpi=300,
vmiddle,nobalance,
loose,
thin,
pilespacing=10pt,%
shortfall=4pt,
]

\ifx\CD@qK\else\CD@PK\relax\CD@FF\CD@e\fi\fi

\CD@vE\CD@hK\else\fi\else\fi\cdrestoreat
\date{ }

\begin{document}

\title{On the Number of Rumer Diagrams}

\author{Valentin Vankov Iliev
\footnote{Section of Algebra and Logic, Institute of Mathematics and
Informatics, Bulgarian Academy of Sciences, 1113 Sofia, Bulgaria.}}

\maketitle

\hangindent 1.4 in \hangafter 0  {\it "In the literature there are a
number of attempts to deal with the matter of skeletal isomerism.
The two papers of widest scope are those of Sylvester, \cite{[55]},
and of Gordan and Alexejeff, \cite{[15]}, in both of which the
parallelism between the chemical theory of molecular structure and
the algebraic theory of invariants is stressed. It seems probable
that interesting and important results might be attained by
exhaustively pursuing this line of attack, but chemists have never
devoted to these papers the attention which they deserve".}

\hangindent 1.4 in \hangafter 0 (A.~C.~Lunn, J.~K.~Senior, 1929,
 \cite{[35]}, p. 1039.)

\begin{abstract}

In this paper we review the famous article "Eine f\"ur die
Valenztheorie geeignete Basis der bin\"aren Vektorinvarianten" of
G.~Rumer, E.~Teller, and H.~Weyl on Rumer diagrams and use a
powerful old result from algebraic geometry to enumerate these
diagrams with given numbers of atoms and valence bonds.

\end{abstract}

\section{Introduction and Notation}

\label{1}

\subsection{Introduction}

\label{1.1}

The paper is devoted to presenting the number of Rumer diagrams with
$n$ atoms which form a molecule with $m$ valence bonds. We follow
the classical article~\cite{[50]} where the atoms are represented by
vectors
\begin{equation}
x^{\left(1\right)}=\left(
\begin{array}{cccccccccccc}
 x_1^{\left(1\right)}\\
 x_2^{\left(1\right)}\\
 \end{array}
\right), \ldots,x^{\left(n\right)}=\left(
\begin{array}{cccccccccccc}
 x_1^{\left(n\right)}\\
 x_2^{\left(n\right)}\\
 \end{array}
\right),\label{1.1.1}
\end{equation}
in the complex plane $\mathbb{C}^2$. When $n\geq 2$ each valence
bond between two of them, $x^{\left(i\right)}$ and
$x^{\left(j\right)}$, is modeled by the determinant
\begin{equation}
[x^{\left(i\right)},x^{\left(j\right)}]=\left|
\begin{array}{cccccccccccc}
 x_1^{\left(i\right)}& x_1^{\left(j\right)}\\
 x_2^{\left(i\right)}& x_2^{\left(j\right)}\\
 \end{array}
\right|,\label{1.1.5}
\end{equation}
viewed as a polynomial in the four variables
$x_1^{\left(i\right)},x_2^{\left(i\right)},
x_1^{\left(j\right)},x_2^{\left(j\right)}$. These determinants,
called \emph{brackets}, considered as complex-valued functions in
two vector arguments $x^{\left(i\right)},x^{\left(j\right)}$,
defined on the plane $\mathbb{C}^2$, are the building blocks of the
set of all such functions
$f(x^{\left(1\right)},\ldots,x^{\left(n\right)})$ which have the
same value (are \emph{invariant}) when changing the coordinate
system in $\mathbb{C}^2$ via transformations with determinant $1$.
All such transformations of $\mathbb{C}^2$ form a group called
\emph{special linear group} and denoted $\sl(\mathbb{C}^2)$, or,
simply $\sl(2)$. More precisely, in accord with the first main
theorem on invariants, homogeneous degree $k$ invariant under the
group $\sl(2)$ polynomial $f$ in
$x_1^{\left(1\right)},x_2^{\left(1\right)},\ldots,
x_1^{\left(n\right)},x_2^{\left(n\right)}$ exists exactly when $k$
is even, $k=2m$, and in this case $f$ is a linear combinations of
products (called \emph{bracket monomials}) of the form
\begin{equation}
[x^{\left(i_1\right)},x^{\left(j_1\right)}]\ldots
[x^{\left(i_m\right)},x^{\left(j_m\right)}],\hbox{\ }1\leq
i_s<j_s\leq n,\hbox{\ }s\in [1,m].\label{1.1.10}
\end{equation}
The expression of any bracket monomial~\eqref{1.1.10} produces a
graph $G$ on $n$ vertices $1,\ldots,n$ and $m$ edges
$(i_1,j_1),\ldots,(i_m,j_m)$ without loops, where the vertices are
placed clockwise on a circle in this order and any edge $(i_s,j_s)$
is represented by the segment with end points $i_s$ and $j_s$.
Conversely, each bracket monomial's expression~\eqref{1.1.10} can be
restituted from its graph $G$ which thus is called its \emph{valence
scheme}. In case there are no intersecting edges, the valence scheme
$G$ is said to be \emph{Rumer diagram}. In quantum chemistry, each
Rumer diagram, or, what is the same, bracket monomial, represents a
pure valence state of a molecule with atoms $1,\ldots,n$ and any
invariant (linear combination of bracket monomials) represents a
valence state (mixture of pure valence states).

This idea goes back to the providential paper~\cite{[55]} of
J.~J.~Sylvester, where he, in particular, wrote: "An invariant of a
form or system of algebraic forms must thus represent a saturated
system of atoms in which the rays of all the atoms are connected
into bonds." A meticulous historic analysis of Sylvester's
contributions, written in a modern language, can be found
in~\cite{[25]}. The analogy between invariant theory and chemistry
was completely ignored by the chemists till the end of 19-th century
and was galvanized by the physicist G.~Rumer in his pioneering
paper~\cite{[45]} and in its popular continuation~\cite{[50]}.

In order for the exposition to be self-contained, in section 2 we
gather the necessary definitions and results from different
mathematical disciplines in three subsections: Algebra, Invariant
theory, and Algebraic geometry. From the content of the first
subsection we note the introduction of a special grading (via
graphical multidegree) of a polynomial ring used later. In the last
two subsections we discuss the isomorphism between the graded ring
consisting of $\sl(2)$-invariant polynomials in
$x_1^{\left(1\right)},x_2^{\left(1\right)},\ldots,
x_1^{\left(n\right)},x_2^{\left(n\right)}$ and the graded
homogeneous coordinate ring of the Grassmannian $G(n,2)$. As a
consequence, we use the famous postulation formula of W.~V.~D.~Hodge
and J.~E.~Littlewood in order to find the dimensions of the
homogeneous components of the graded ring of invariants in
Theorem~\ref{5.10.10}.

Section 3 is devoted to meticulous proofs of the three theorems
from~\cite{[50]} and to the enumeration of Rumer diagrams with given
numbers of atoms and valence bonds presented in
Theorem~\ref{10.5.1}.

\subsection{Notation}

\label{1.5}

Below we introduce some notation.

$\mathbb{N}$: The set of positive integers.

$\mathbb{N}_+$: The set of nonnegative integers.

$\mathbb{C}$: The field of complex numbers.

$[1,k]=\{1,\ldots,k\}\subset\mathbb{N}$.

$\binom{k}{\ell}$ for $k\in \mathbb{N}$, $\ell\in \mathbb{N}_+$: The
binomial coefficient.

$n\in\mathbb{N}$, $n\geq 2$: The number of atoms.

$m\in\mathbb{N}_+$: The number of valence bonds.

\section{Mathematical Background}

\label{5}

\subsection{Algebra}

\label{5.1}

Some elementary notions and statements from algebra are supposed to
be known. We suggest~\cite{[1]} and~\cite{[5]} as universal
references.

By a \emph{monoid} we mean a set $\Delta$ together with an
associative binary operation on $\Delta$, written additively, which
has zero element $0$. Examples are the set $\mathbb{N}_+$ with
addition of integers and any Cartesian power $\mathbb{N}_+^\ell$
with componentwise addition of $\ell$-tuples.

We denote by $A=\mathbb{C}[z_1,\ldots,z_r]$ the ring of polynomials
in the variables $z_1,\ldots,z_r$ over the field $\mathbb{C}$ of
complex numbers. Given a monoid $\Delta$, we say that $A$ is a
\emph{graded of type $\Delta$ ring} if there exists a family
$(A_\lambda)_{\lambda\in\Delta}$ of subgroups of the additive group
$A$ such that $A=\oplus_{\lambda\in\Delta} A_\lambda$ (any
polynomial $f\in A$ can be presented uniquely as
$f=\sum_{\lambda\in\Delta}f_\lambda$ , $f_\lambda\in A_\lambda$, and
only a finite number of $f_\lambda$ are nonzero) and $A_\lambda
A_\mu\subset A_{\lambda+\mu}$ for $\lambda,\mu\in\Delta$. A subring
$I\subset A$ is called \emph{graded subring} if
$I=\oplus_{\lambda\in\Delta}I_\lambda$ where $I_\lambda=I\cap
A_\lambda$ for all $\lambda\in\Delta$. A nonempty subset $I\subset
A$ is called \emph{ideal} of $A$ if difference of two polynomials in
$I$ is again in $I$ and if the product of any polynomial with a
member of $I$ is again a member of $I$. The ideal $I$ in $A$ is said
to be \emph{homogeneous} if $I$ is a graded subring. We note that an
ideal is homogeneous if and only if it can be generated by
homogeneous elements.

Any ideal $I$ in the ring $A$ produces a ring $A/I$, called
\emph{factor-ring}, in the following way: The members of the set
$A/I$ are all subsets of $A$ of the form $f+I$, $f\in A$, and we
have $f+I=g+I$ if and only if $f-g\in I$. Addition and
multiplication in $A/I$ are defined naturally: $(f+I)+(g+I)=f+g+I$,
$(f+I)(g+I)=fg+I$, $0+I=I$ is the zero element, and $1+I$ is the
unit.

In general, a map from a ring (graded type $\Delta$ ring) to another
ring (graded type $\Delta$ ring) is called \emph{homomorphism}
(\emph{homogeneous homomorphism}) if it maps sums and products onto
the sums and products of the corresponding images, respectively (and
maps any homogeneous component into the corresponding homogeneous
component). \emph{Kernel} of this homomorphism is the ideal
consisting of members of the domain which are mapped onto zero. When
the homomorphism is homogeneous, then its kernel turns out to be
homogeneous ideal. The map
\[
c\colon A\to A/I,\hbox{\ }f\mapsto f+I,
\]
is said to be the \emph{canonical homomorphism} and its kernel is
$I$. If the ring $A$ is a graded of type $\Delta$,
$A=\oplus_{\lambda\in\Delta} A_\lambda$, and if the ideal $I$ is
homogeneous, then $A/I$ inherits this grading via $c$:
$A/I=\oplus_{\lambda\in\Delta}c(A_\lambda)$. In this case the
canonical homomorphism is homogeneous. In general, any surjective
homomorphism $\varphi\colon A\to B$ of rings is, up to a unique
isomorphism produced by $\varphi$, canonical: In the commutative
diagram below $\bar{\varphi}$ is an isomorphism.
\begin{equation}
\begin{diagram}
A           &         & \\
\dTo_{c} & \rdTo^{\varphi}& \\
A/I        & \rTo^{\bar{\varphi}} & B.\\\label{5.1.15}
\end{diagram}
\end{equation}
Here the kernel of $\varphi$ is $I$ and commutativity means
$\varphi=\bar{\varphi}\circ c$.

The \emph{natural grading} of $A$ can be defined as the grading of
type $\mathbb{N}_+$, where $A_\lambda$ is the subgroup of $A$,
consisting of all homogeneous degree $\lambda$ polynomials,
$\lambda\in\mathbb{N}_+$, plus the zero polynomial (which has degree
$-\infty$). Since $A_0=\mathbb{C}$, we obtain that any
\emph{homogeneous component} $A_\lambda$ is a $\mathbb{C}$-linear
space.

Let $y_{ij}$, $1\leq i<j\leq n$, be $\binom{n}{2}$ in number
variables, and let $\mathbb{C}[y_{ij}]=\mathbb{C}[y_{ij},\hbox{\
}1\leq i<j\leq n]$ be the corresponding  polynomial ring. Let us set
$B=\mathbb{C}[y_{ij}]$ for short and let
$B=\oplus_{m\in\mathbb{N}_+}B_m$ be the natural grading. For any
monomial
\begin{equation}
g=y_{i_1j_1}\ldots y_{i_mj_m}\in B_m,\hbox{\ } 1\leq i_s<j_s\leq
n,\hbox{\ } s\in [1,m],\label{5.1.1}
\end{equation}
a graph $G$ on $n$ vertices $1,\ldots,n$ and $m$ edges
$(i_1,j_1),\ldots,(i_m,j_m)$ without loops can be constructed
exactly in the same way lake the graph of a bracket
monomial~\eqref{1.1.10}. Conversely, each such monomial $g$ can be
restituted from its graph $G$ and the correspondence is one-to-one.
We define $\deg_v(g)$ to be the degree of the vertex $v$ of the
graph $G$, $v\in [1,n]$.

Given $m,m_1,\ldots,m_n\in\mathbb{N}_+$ with $m_1+\cdots+m_n=2m$, we
denote by $B_{m_1,\ldots,m_n}$ the $\mathbb{C}$-linear subspace of
$B_m$, spanned by all monomials~\eqref{5.1.1} such that
$\deg_v(g)=m_v$, $v\in [1,n]$. The $n$-tuple $m_1,\ldots,m_n$ is
called \emph{graphical multidegree} of the monomial~\eqref{5.1.1}.

Let us denote by $\Delta_2^n$ the monoid  consisting of all
$n$-tuples $(m_1,\ldots,m_n)\in \mathbb{N}_+^n$ with even sum. Thus,
we obtain a finer grading of type $\Delta_2^n$ of the polynomial
ring $B=\mathbb{C}[y_{ij}]$:
\begin{equation}
B=\oplus_{m\in\mathbb{N}_+}\oplus_{m_1+\cdots+m_n=2m}
B_{m_1,\ldots,m_n},\label{5.1.5}
\end{equation}
and, moreover, for any $m\in\mathbb{N}_+$ we have
\begin{equation}
B_m=\oplus_{m_1+\cdots+m_n=2m} B_{m_1,\ldots,m_n}.\label{5.1.10}
\end{equation}

\subsection{Invariant Theory}

\label{5.5}

The vectors in $\mathbb{C}^2$ have the form~\eqref{1.1.1} (that is,
are $2\times 1$-matrices) and the invertible linear transformations
of $\mathbb{C}^2$ (changes of coordinates) are $2\times 2$-matrices
$\sigma$ with $\det\sigma\neq 0$. The latter form a group $\gl(2)$
(\emph{general linear group}) and the special linear group $\sl(2)$
is the subgroup of $\gl(2)$, consisting of all $\sigma\in\gl(2)$
with $\det\sigma=1$.

For general vectors~\eqref{1.1.1} we set
\[
\mathbb{C}[x^{\left(1\right)},\ldots, x^{\left(n\right)}]=
\mathbb{C}[x_1^{\left(1\right)},x_2^{\left(1\right)},\ldots,
x_1^{\left(n\right)},x_2^{\left(n\right)}]
\]
to be the polynomial ring in $2n$ variables furnished with the
natural grading:
\begin{equation}
\mathbb{C}[x^{\left(1\right)},\ldots,
x^{\left(n\right)}]=\oplus_{k\in\mathbb{N}_+}\mathbb{C}[x^{\left(1\right)},\ldots,
x^{\left(n\right)}]_k. \label{1.5.1}
\end{equation}

The group $\gl(2)$ (and hence $\sl(2)$) acts on the plane
$\mathbb{C}^2$ by matrix multiplication, called \emph{left
multiplication}: If
\[
\sigma=\left(
\begin{array}{cccccccccccc}
 \sigma_{11} & \sigma_{12} \\
 \sigma_{21} & \sigma_{22} \\
 \end{array}
\right)\hbox{\rm\ and\ }x=\left(
\begin{array}{cccccccccccc}
 x_1\\
 x_2\\
 \end{array}
\right),
\]
then $\sigma\cdot x=\sigma x$. This action can be extended naturally
on the Cartesian product
$\mathbb{C}^2\times\cdots\times\mathbb{C}^2$ ($n$ times) via the
rule $\sigma\cdot (x^{\left(1\right)},\ldots,
x^{\left(n\right)})=(\sigma\cdot x^{\left(1\right)},\ldots,
\sigma\cdot x^{\left(n\right)})$. Thus, we obtain an action of the
group $\gl(2)$ on the polynomial ring
$\mathbb{C}[x^{\left(1\right)},\ldots, x^{\left(n\right)}]$ via the
formula
\begin{equation}
(\sigma\cdot f) (x^{\left(1\right)},\ldots,
x^{\left(n\right)})=f(\sigma^{-1}\cdot x^{\left(1\right)},\ldots,
\sigma^{-1}\cdot x^{\left(n\right)}).\label{5.5.1}
\end{equation}

All of the above actions are \emph{linear} ones: The left
multiplication by a fixed $\sigma\in\gl(2)$ induces an invertible
linear transformation of the corresponding $\mathbb{C}$-linear
space. In general, any linear action of a group $\Gamma$ on a
finite-dimensional linear space $W$ is called
\emph{finite-dimensional representation} of this group. In case $W$
can not be represented as a direct sum of two nonzero subspaces on
which this action of $\Gamma$ induces actions on these subspaces,
the representation $W$ of $\Gamma$ is said to be \emph{irreducible}.
When for any $\sigma\in\Gamma$, the corresponding left
multiplication by $\sigma$ in $W$ is the identity linear
transformation, the representation is called \emph{identical}.

Any finite-dimensional representation of $\gl(2)$ or $\sl(2)$ can be
represented as a finite direct sum of irreducible representations of
this group.

In order to explain the relation of the ring
$\mathbb{C}[x^{\left(1\right)},\ldots, x^{\left(n\right)}]$ with the
main statements in the classical invariant theory, we need to
furnish it with a finer structure --- this of a graded
ring~\eqref{1.5.1}. The action of the group $\gl(2)$ preserves the
grading: If $f\in \mathbb{C}[x^{\left(1\right)},\ldots,
x^{\left(n\right)}]_k$, then $\sigma\cdot f\in
\mathbb{C}[x^{\left(1\right)},\ldots, x^{\left(n\right)}]_k$.
Moreover, let $f\neq 0$, $f=\sum_{n\geq 0}f_k$. The polynomial $f$
is invariant if and only if each homogeneous polynomial $f_k$ is
invariant (see~\cite[Ch. 1, Sect. 2, Proposition 2]{[10]}).

The determinants (brackets)
$p_{ij}=[x^{\left(i\right)},x^{\left(j\right)}]$ from~\eqref{1.1.5}
for $1\leq i<j\leq n$ are members of the polynomial ring
$\mathbb{C}[x^{\left(1\right)},\ldots, x^{\left(n\right)}]$. Let us
define $p_{ij}=-p_{ji}$ for $i>j$. The homogeneous degree $2$
polynomials $p_{ij}$, $1\leq i<j\leq n$, satisfy the following
quadratic identities:
\begin{equation}
p_{i_1j_1}p_{j_2j_3}-p_{i_1j_2}p_{j_1j_3}+p_{i_1j_3}p_{j_1j_2}=0
\label{5.5.10}
\end{equation}
for all subsets $\{i_1,j_1,j_2,j_3\}\subset [1,n]$, consisting of
four elements. Let $\inv^{\left(n\right)}$ be the subring of
$\mathbb{C}[x^{\left(1\right)},\ldots, x^{\left(n\right)}]$,
generated by $p_{ij}$, $1\leq i<j\leq n$. The members of the ring
$\inv^{\left(n\right)}$ are constructed in the following way: First,
we define the map
\begin{equation}
\varphi\colon\mathbb{C}[y_{ij}]
\to\mathbb{C}[x^{\left(1\right)},\ldots, x^{\left(n\right)}],\hbox{\
}f(y_{ij})\mapsto f(p_{ij}).\label{5.5.15}
\end{equation}
This map is a homomorphism of rings and the ring
$\inv^{\left(n\right)}$ is the image of $\varphi$ (that is,
$\inv^{\left(n\right)}=\{f(p_{ij})\mid f\in \mathbb{C}[y_{ij}]\}$).
Moreover, $\inv^{\left(n\right)}$ inherits the natural structure of
graded ring on the polynomial ring $\mathbb{C}[y_{ij}]$:
$\inv^{\left(n\right)}=\oplus_{m\in\mathbb{N}_+}\inv_m^{\left(n\right)}$,
where $\inv_m^{\left(n\right)}=\varphi(\mathbb{C}[y_{ij}]_m)$. Thus,
for any $m\in\mathbb{N}_+$ the $\mathbb{C}$-linear space
$\inv_m^{\left(n\right)}$ is spanned by all bracket monomials
\begin{equation}
p_{i_1j_1}\ldots p_{i_mj_m}, 1\leq i_s<j_s\leq n,\hbox{\ } s\in
[1,m].\label{5.5.20}
\end{equation}

The first main theorem of invariants (see, for example~\cite[Ch. 1,
Sect. 2, Proposition 3 and Ch. 2, Sect. 5, Theorem]{[10]}, or,
\cite[1. Fundamentalsatz]{[50]}) yields that the $\mathbb{C}$-linear
space $\inv_m^{\left(n\right)}$ consists of all homogeneous
$\sl(2)$-invariant polynomials of degree $k=2m$ in
$\mathbb{C}[x^{\left(1\right)},\ldots, x^{\left(n\right)}]$ and
there are no homogeneous invariants in
$\mathbb{C}[x^{\left(1\right)},\ldots, x^{\left(n\right)}]_{2m+1}$,
$m\in\mathbb{N}_+$. Thus, the subring
$\inv^{\left(n\right)}\subset\mathbb{C}[x^{\left(1\right)},\ldots,
x^{\left(n\right)}]$ consists of all $\sl(2)$-invariant polynomials.

In case $n=1$, let us denote by $\inv^{\left(1\right)}$ the subring
of $\mathbb{C}[x^{\left(1\right)}]$, consisting of all
$\sl(2)$-invariant polynomials in two variables
$x_1^{\left(1\right)},x_2^{\left(1\right)}$. We have
$\inv^{\left(1\right)}=\mathbb{C}$, and, more precisely,
$\inv_0^{\left(1\right)}=\mathbb{C}$, $\inv_m^{\left(1\right)}=0$
for $m\in\mathbb{N}$.

\subsection{Algebraic Geometry: Grassmanians}

\label{5.10}

Let us consider the homogeneous ideal $Q$ in the polynomial ring
$\mathbb{C}[y_{ij}]$, generated by the homogeneous degree $2$
polynomials
\begin{equation}
Q_{i_1;j_1,j_2,j_3}=y_{i_1j_1}y_{j_2j_3}-y_{i_1j_2}y_{j_1j_3}+
y_{i_1j_3}y_{j_1j_2}\label{5.10.1}
\end{equation}
for all subsets $\{i_1,j_1,j_2,j_3\}\subset [1,n]$, consisting of
four elements. Because of~\eqref{5.5.10} and of the second main
theorem of invariants (\cite[2. Fundamentalsatz]{[50]}), the kernel
of the (graded) homomorphism~\eqref{5.5.15} is the homogeneous ideal
$Q$ of $\mathbb{C}[y_{ij}]$. Thus, in accord with~\eqref{5.1.15},
the graded ring
$\inv^{\left(n\right)}=\oplus_{m\in\mathbb{N}_+}\inv_m^{\left(n\right)}$
and the graded factor-ring $C=\mathbb{C}[y_{ij}]/Q$,
$C=\oplus_{m\in\mathbb{N}_+}C_m$, are isomorphic. In particular, for
any $m\in\mathbb{N}_+$ the $\mathbb{C}$-linear spaces
$\inv_m^{\left(n\right)}$ and $C_m$ are isomorphic.

Let us consider the finer grading of the polynomial ring
$\mathbb{C}[y_{ij}]$ from~\eqref{5.1.5} via graphical multidegrees:
\[
\mathbb{C}[y_{ij}]\oplus_{m\in\mathbb{N}_+}\oplus_{m_1+\cdots+m_n=2m}
\mathbb{C}[y_{ij}]_{m_1,\ldots,m_n}.
\]
Since any polynomial $Q_{i_1;j_1,j_2,j_3}\in \mathbb{C}[y_{ij}]_2$
has graphical multidegree $m_1,\ldots,m_n$ where $m_s=1$ for $s\in
\{i_1,j_1,j_2,j_3\}$ and $m_s=0$ otherwise, the ideal $Q$ is also
homogeneous relative to this finer grading. Therefore the
factor-ring $C=\mathbb{C}[y_{ij}]/Q$ inherits also this grading. In
other words, according to~\eqref{5.1.10}, we can write
\[
C=\oplus_{m\in\mathbb{N}_+}\oplus_{m_1+\cdots+m_n=2m}C_{m_1,\ldots,m_n},\hbox{\
} C_m=\oplus_{m_1+\cdots+m_n=2m}C_{m_1,\ldots,m_n},
\]
where $C_m$ is the image of $\mathbb{C}[y_{ij}]_m$ and
$C_{m_1,\ldots,m_n}$ is the image of
$\mathbb{C}[y_{ij}]_{m_1,\ldots,m_n}$ via the canonical homomorphism
$\mathbb{C}[y_{ij}]\to \mathbb{C}[y_{ij}]/Q$.

\begin{remark} \label{5.10.5}  {\rm Let $X$ be the
projective algebraic variety in the projective space
$\mathbb{P}^{\binom{n}{2}-1}$, defined by the equations
$Q_{i_1;j_1,j_2,j_3}=0$ for all polynomials~\eqref{5.10.1}. The
structure of $X$ can be recovered completely by the structure of its
\emph{homogeneous coordinate ring} $C$, that is, by the graded ring
$\inv^{\left(n\right)}=\oplus_{m\in\mathbb{N}_+}\inv_m^{\left(n\right)}$
of $\sl(2)$-invariants in $\mathbb{C}[x^{\left(1\right)},\ldots,
x^{\left(n\right)}]$.

On the other hand, let $G(2,n)$ be the set of all planes in the
$\mathbb{C}$-linear space $C^n$ through the origin. There exists an
one-to-one map between $X$ and $G(2,n)$ via which we can furnish the
set $G(2,n)$ with a structure of projective algebraic variety called
\emph{Grassmann variety}, or, \emph{Grassmannian} for short. It
turns out that $G(2,n)$ is an irreducible smooth algebraic variety
of dimension $2n-4$ with homogeneous coordinate ring
$C=\oplus_{m\in\mathbb{N}_+}C_m$. }

\end{remark}

The dimension of the $\mathbb{C}$-linear space $C_m$,
$m\in\mathbb{N}$, is evaluated as a polynomial of degree $2n-4$ in
one variable $m$ with rational coefficients by W.~V.~D.~Hodge and
J.~E.~Littlewood (a particular case of their so called
\emph{postulation formula} which can be found, for example, in the
book~\cite[vol. 2, Ch. XIV,\P 9]{[30]}). Since $C_m$ and
$\inv_m^{\left(n\right)}$ are isomorphic $\mathbb{C}$-linear spaces
for all $m\in\mathbb{N}_+$, we obtain the following:

\begin{theorem}\label{5.10.10} For any $m\in\mathbb{N}$ one has
\[
\dim(\inv_m^{\left(n\right)})=\left|\begin{array}{ccc}
\binom{m+n-1}{n-1} & \binom{m+n-2}{n-1}  \\
\binom{m+n-1}{n-2} & \binom{m+n-2}{n-2}  \\
\end{array}\right|.
\]

\end{theorem}

\begin{corollary}\label{5.10.15} If $n=2$, then
$\dim(\inv_m^{\left(n\right)})=1$ for all $m\in\mathbb{N}$.

\end{corollary}

\begin{corollary}\label{5.10.20} If $n\geq 3$, then
\[
\dim(\inv_m^{\left(n\right)})=
\frac{1}{(n-1)!(n-2)!}(m+1)(m+n-1)\prod_{\iota=2}^{n-2}(m+\iota)^2
\]
for all $m\in\mathbb{N}$.

\end{corollary}

\begin{remark}\label{5.10.35} {\rm The author learned about the
isomorphism between the graded ring $\inv^{\left(n\right)}$ of
invariants of the special linear group $\sl(2)$ and the homogeneous
coordinate ring $C$ of the Grassmannian $G(2,n)$ from the comments
of V~L.~Popov to the Russian translation of~\cite{[50]} in the
book~\cite{[40]}.

}

\end{remark}

\section{Rumer Diagrams}

\label{10}

\subsection{The Results of G.~Rumer, E.~Teller, and H.~Weyl}

\label{10.1}

Let $R(m)$ be the set of all Rumer diagrams produced by a fixed
set~\eqref{1.1.1} of $n$ atoms connected with $m$ valence bonds.
Now, let us attach, in addition, to each atom $x^{\left(i\right)}$,
$1\in [1,n]$, the number $m_i$, $m_i\in\mathbb{N}_+$, of valence
bonds connecting this atom. Given a sequence of non-negative
integers $m_1,\ldots,m_n$, let $R(m_1,\ldots,m_n)$ be the set of all
Rumer diagrams with these prescribed valences and let
$\rho(m_1,\ldots,m_n)$ be their number. Let
$\inv_{m_1,\ldots,m_n}^{\left(n\right)}$ be the $\mathbb{C}$-linear
subspace of $\inv_m^{\left(n\right)}$ spanned by all bracket
monomials corresponding to Rumer diagrams from $R(m_1,\ldots,m_n)$,
where $m_1+\cdots+m_n=2m$. We have
\[
R(m)=\cup_{m_1+\cdots+m_n=2m} R(m_1,\ldots,m_n)
\]
--- a union of pair-wise disjoint sets $R(m_1,\ldots,m_n)$, where for
given $m$ the sequence $m_1,\ldots,m_n$ runs through all
non-negative solutions of the Diophantine equation
$m_1+\cdots+m_n=2m$.

In their paper~\cite{[50]}, G.~Rumer, E.~Teller, and H.~Weyl
obtained the theorems below. For completeness, for each one we give
an outline of its proof. Below we use the following notions from
graph theory: For any two vertices $i$,$j$ of a valence scheme $G$
there are two \emph{arcs} on the circle with these endpoints. We
define \emph{length} of an arc as the number of its internal
non-isolated vertices plus one.

\begin{theorem}\label{10.1.1} The bracket monomials, corresponding
to the Rumer diagrams from $R(m)$ span the $\mathbb{C}$-linear space
$\inv_m^{\left(n\right)}$.

\end{theorem}

\begin{proof} We proceed by induction on $m$. The case $m=1$ holds
because the $\mathbb{C}$-linear space $\inv_1^{\left(n\right)}$ is
spanned by the brackets $p_{ij}$, $1\leq i<j\leq n$, whose valence
schemes are exactly the Rumer diagrams from $R(1)$. Now, assume that
the statement is true for $m-1$. Let
\begin{equation}
p_{i_1j_1}\ldots p_{i_mj_m}\label{10.1.5}
\end{equation}
be the expression of the bracket monomial with valence scheme $G$
and without loss of generality we can assume that one of the arcs of
the edge $(i,j)$, $i=i_1$, $j=j_1$, is (one of) the arcs of $G$ with
minimal length. In case this minimal length is $1$, we represent the
bracket monomial $p_{i_2j_2}\ldots
p_{i_mj_m}\in\inv_{m-1}^{\left(n\right)}$ as a linear combination of
bracket monomials in $\inv_{m-1}^{\left(n\right)}$, whose valence
schemes $G,G',\ldots$ are Rumer diagrams from $R(m-1)$. Multiplying
the last equality by $p_{ij}$, we obtain that $p_{i_1j_1}\ldots
p_{i_mj_m}$ is a linear combination of bracket monomials in
$\inv_m^{\left(n\right)}$, whose valence schemes are Rumer diagrams
from $R(m)$ because the attaching the edge $(i,j)$ to the Rumer
diagrams $G,G',\ldots$ produces Rumer diagrams. Now, let this
minimal length $g$ be at least $2$ and let $k$ be an internal
non-isolated point. Since there is no edge of $G$, which connects
$k$ to $i$ and $k$ to $j$, the edge of $G$ that passes through $k$
intersects $(i,j)$. Let $\ell$ be the other endpoint. We can suppose
$\{k,\ell\}=\{i_2,j_2\}$ and then, using the
relation~\eqref{5.5.10}, we obtain
\[
\pm p_{i_1j_1}\ldots p_{i_mj_m}=p_{ij}p_{k\ell}p_{i_3j_3}\ldots
p_{i_mj_m}=
\]
\[
p_{ik}p_{j\ell}p_{i_3j_3}\ldots
p_{i_mj_m}-p_{i\ell}p_{jk}p_{i_3j_3}\ldots p_{i_mj_m}.
\]
Thus, the bracket monomial~\eqref{10.1.5} is a linear combination of
two bracket monomials such that their valence schemes have arcs with
length strictly less than $g$. By repeating this procedure a finite
number of times on the summands, we finish the proof.

\end{proof}

In the next two theorems we permit $n$ to assume also value $1$.

\begin{theorem}\label{10.1.10} The dimensions $N(m_1,\ldots,m_n)=\dim
\inv_{m_1,\ldots,m_n}^{\left(n\right)}$ satisfies the recurrence
relation
\[
N(m_1,\ldots,m_n,m_{n+1})=\sum_{\mu_n}N(m_1,\ldots,m_{n-1},\mu_n),
\]
\[
N(m_1)=\left\{\begin{array}{ccc}
1 & \hbox{\rm\ if\ } & m_1=0  \\
0 & \hbox{\rm\ if\ } & m_1\geq 1,  \\
\end{array}\right.
\]
where $\mu_n\in\mathbb{N}_+$ runs through those members of the
sequence $m_n+m_{n+1},m_n+m_{n+1}-2,\ldots,|m_n-m_{n+1}|$, such that
$m_n,m_{n+1},\mu_n$ form an \emph{even triangle}, that is,
$\mu_n\leq m_n+m_{n+1}$, $m_n\leq \mu_n+m_{n+1}$, $m_{n+1}\leq
\mu_n+m_n$, and the sum $m_n+m_{n+1}+\mu_n$ is an even number.

\end{theorem}

\begin{proof} The group $\sl(2)$ acts on the polynomial ring
$\mathbb{C}[x^{\left(1\right)}]=
\mathbb{C}[x_1^{\left(1\right)},x_2^{\left(1\right)}]$ via the
rule~\eqref{5.5.1} and for any $k\in\mathbb{N}_+$ the homogeneous
components $\mathbb{C}[x^{\left(1\right)}]_k$ are stable. Therefore
we obtain an action of $\sl(2)$ on any $\mathbb{C}$-linear space
$\mathbb{C}[x^{\left(1\right)}]_k$ of dimension $k+1$, or, what is
the same, a representation of the group $\sl(2)$. According to, for
example,~\cite[Sec. 2, 2.3.1]{[20]}, the last representation is
irreducible. In particular,the group $\sl(2)$ acts trivially on
$\mathbb{C}[x^{\left(1\right)}]_0=\mathbb{C}$:
$\sigma\cdot\alpha=\alpha$ for all $\alpha\in\mathbb{C}$ and the
corresponding representation is the identical representation.

Given $m_1,m_2\in\mathbb{N}_+$, in accord with Clebsch-Gordan
formula (see, for example,~\cite[Sec. 7, 7.1.4, Ex. 2]{[20]}) we
have
\[
\mathbb{C}[x^{\left(1\right)}]_{m_1}\otimes
\mathbb{C}[x^{\left(1\right)}]_{m_2}=\oplus_{\mu}
\mathbb{C}[x^{\left(1\right)}]_\mu,
\]
where the index $\mu\in\mathbb{N}_+$ runs through all integers
$m_1+m_2,m_1+m_2-2,\ldots, |m_1-m_2|$. In other words, $m_1,m_2,\mu$
are side lengths of a triangle and, moreover, the perimeter of this
triangle is an even number (\emph{even triangle}). We denote this
symmetric ternary relation by $2\bigtriangleup(m_1,m_2,\mu)$ and can
write down for short
\begin{equation}
\mathbb{C}[x^{\left(1\right)}]_{m_1}\otimes
\mathbb{C}[x^{\left(1\right)}]_{m_2}=
\oplus_{2\bigtriangleup\left(m_1,m_2,\mu\right)}
\mathbb{C}[x^{\left(1\right)}]_\mu.\label{10.1.15}
\end{equation}

Note that the right hand side of decomposition~\eqref{10.1.15}
contains the identical representation if and only if $m_1=m_2$, and,
in this case all other summands are different from it.

Tensoring the equality~\eqref{10.1.15} by
$\mathbb{C}[x^{\left(1\right)}]_{m_3}$, the equality thus obtained
by $\mathbb{C}[x^{\left(1\right)}]_{m_4}$, etc., we obtain
inductively on $n$ that
\begin{equation}
\mathbb{C}[x^{\left(1\right)}]_{m_1}\otimes\cdots\otimes
\mathbb{C}[x^{\left(1\right)}]_{m_n}=\oplus_{\mu\geq 0}
N_\mu(m_1,\ldots,m_n)\mathbb{C}[x^{\left(1\right)}]_\mu
\label{10.1.20}
\end{equation}
where $N_\mu(m_1,\ldots,m_n)\geq 0$ and only for a finite number of
indices $\mu$ we have $N_\mu(m_1,\ldots,m_n)>0$. The corresponding
$\mathbb{C}[x^{\left(1\right)}]_\mu$ are the irreducible components
of the representation
$\mathbb{C}[x^{\left(1\right)}]_{m_1}\otimes\cdots\otimes
\mathbb{C}[x^{\left(1\right)}]_{m_n}$ and $N_\mu(m_1,\ldots,m_n)>0$
are their multiplicities.

Now let $n\geq 2$, let us write down equality~\eqref{10.1.20} for
$n-1$ and tensor it by $\mathbb{C}[x^{\left(1\right)}]_{m_n}$:
\begin{equation}
\mathbb{C}[x^{\left(1\right)}]_{m_1}\otimes\cdots\otimes
\mathbb{C}[x^{\left(1\right)}]_{m_n}=\oplus_{\mu\geq 0}
N_\mu(m_1,\ldots,m_{n-1})\mathbb{C}[x^{\left(1\right)}]_\mu\otimes
\mathbb{C}[x^{\left(1\right)}]_{m_n}. \label{10.1.25}
\end{equation}
The identical representation is contained in the right hand side
of~\eqref{10.1.25} if and only if $\mu=m_n$, and in this case we
have
\begin{equation}
N_{m_n}(m_1,\ldots,m_{n-1})=N_0(m_1,\ldots,m_n)=N(m_1,\ldots,m_n)
\label{10.1.30}
\end{equation}
for all $n\geq 2$.

Further, let us plug the equality
\[
\mathbb{C}[x^{\left(1\right)}]_\mu\otimes
\mathbb{C}[x^{\left(1\right)}]_{m_n}=
\oplus_{2\bigtriangleup\left(\mu,m_n,\kappa\right)}
\mathbb{C}[x^{\left(1\right)}]_\kappa
\]
in~\eqref{10.1.25}:
\[
\mathbb{C}[x^{\left(1\right)}]_{m_1}\otimes\cdots\otimes
\mathbb{C}[x^{\left(1\right)}]_{m_n}= \oplus_{\nu\geq 0}
N_\nu(m_1,\ldots,m_{n-1})
(\oplus_{2\bigtriangleup\left(\nu,m_n,\kappa\right)}
\mathbb{C}[x^{\left(1\right)}]_\kappa).
\]
Let us fix $\mu$ and set $\kappa=\mu$. Then
\[
\mathbb{C}[x^{\left(1\right)}]_{m_1}\otimes\cdots\otimes
\mathbb{C}[x^{\left(1\right)}]_{m_n}=
\]
\begin{equation}
(\sum_{2\bigtriangleup\left(\nu,m_n,\mu\right)}N_\nu(m_1,\ldots,m_{n-1}))
\mathbb{C}[x^{\left(1\right)}]_\mu
+\oplus_{\nu\neq\mu}\mathbb{C}[x^{\left(1\right)}]_\nu.\label{10.1.35}
\end{equation}
Comparing the equalities~\eqref{10.1.20} and~\eqref{10.1.35}, we
have
\[
N_\mu(m_1,\ldots,m_n)=
\sum_{2\bigtriangleup\left(m_n,\mu,\nu\right)}N_\nu(m_1,\ldots,m_{n-1}).
\]
Taking into account~\eqref{10.1.30}, we obtain
\[
N(m_1,\ldots,m_n,m_{n+1})=
\sum_{2\bigtriangleup\left(m_n,m_{n+1},\mu_n\right)}N(m_1,\ldots,m_{n-1},\mu_n).
\]
In case $n=1$ we have
\[
N(m_1)=\left\{\begin{array}{ccc}
1 & \hbox{\rm\ if\ } & m_1=0  \\
0 & \hbox{\rm\ if\ } & m_1\geq 1.  \\
\end{array}\right.
\]
The proof is done.

\end{proof}

\begin{remark}\label{10.1.40} {\rm In quantum mechanics the
irreducible representation of the special linear group $\sl(2)$,
given by the formula
\[
\sl(2)\to\gl(\mathbb{C}[x_1,x_2]_m),\hbox{\
}\sigma\mapsto(f(x)\mapsto f(\sigma^{-1}\cdot x)),
\]
where $f\in\mathbb{C}[x_1,x_2]_m$, $x=\left(
\begin{array}{cccccccccccc}
 x_1\\
 x_2\\
 \end{array}
\right)$, models the spin momentum $\frac{m}{2}$ of the electron
configuration of an atom of valence $m$.

Clebsch-Gordan formula
\[
\mathbb{C}[x_1,x_2]_{m_1}\otimes
\mathbb{C}[x_1,x_2]_{m_2}=\oplus_{\mu} \mathbb{C}[x_1,x_2]_\mu,
\]
where the index $\mu\geq 0$ runs through all integers
$m_1+m_2,m_1+m_2-2,\ldots, |m_1-m_2|$, reflects the following fact
from quantum theory: If the electron configuration of an atom of
valence $m_1$ interacts with the electron configuration of an atom
of valence $m_2$, then the resulting spin $\frac{\mu}{2}$ can assume
all values between $\frac{|m_1-m_2|}{2}$ and $\frac{m_1+m_2}{2}$.

}

\end{remark}

\begin{theorem}\label{10.1.45} {\rm (i)} The function
$\rho(m_1,\ldots,m_n)$ satisfies the recurrence relation from
Theorem~\ref{10.1.10}.

{\rm (ii)}  One has $\rho(m_1,\ldots,m_n)=N(m_1,\ldots,m_n)$ for all
$m_1,\ldots,m_n\in\mathbb{N}_+$.

\end{theorem}

\begin{proof} {\rm (i)} Given a valence scheme
in $V(m_1,\ldots,m_n,m_{n+1})$, let $m_{n,n+1}$ be the number of
valence bonds that connect $n$ and $n+1$. We construct a valence
scheme in $V(m_1,\ldots,m_{n-1},\mu_n)$, where
$\mu_n=m_n+m_{n+1}-2m_{n,n+1}$ (so
$2\bigtriangleup(m_n,m_{n+1},\mu_n)$ holds) on the vertices
$1,\ldots,a$ in the following way:

(a) Omit the valence bonds that connect $n$ and $n+1$.

(b) Omit the vertex $n+1$ and attach its $m_{n+1}-m_{n,n+1}$
vertices left, to the vertex $n$ without changing their other
endpoints.

Thus, we obtain a map
\[
\psi\colon V(m_1,\ldots,m_n,m_{n+1})\to
\cup_{2\bigtriangleup\left(m_n,m_{n+1},
\mu_n\right)}V(m_1,\ldots,m_{n-1},\mu_n),
\]
which turns out to be surjective.

Indeed, given a valence scheme in $V(m_1,\ldots,m_{n-1},\mu_n)$ we
can find a valence scheme in $V(m_1,\ldots,m_{n-1},m_n,m_{n+1})$,
which is mapped onto the former, in the following way:

(A) Add an additional point $n+1$ to $1,\ldots,n$.

(B) Distribute the $\mu_n=m_n+m_{n+1}-2r$ valence bonds through $n$
where $0\leq r\leq\min(m_n,m_{n+1})$ (that is,
$2\bigtriangleup(m_n,m_{n+1},\mu_n)$) as follows: Leave $m_n-r$ in
number valence bonds attached to $n$ and attach $m_{n+1}-r$ in
number valence bonds to $n+1$.

(C) Add $r$ valence bonds which connect $n$ with $n+1$.

The inverse image $\psi^{-1}(G)$ of a valence scheme $G\in
V(m_1,\ldots,m_{n-1},\mu_n)$, where
$2\bigtriangleup(m_n,m_{n+1},\mu_n)$, contains at most
$\binom{\mu_n}{m_{n+1}-r}$ valence schemes from
$V(m_1,\ldots,m_n,m_{n+1})$.

Since $n$ and $n+1$ are neighbours (the arc $\widehat{n\hbox{\
}n+1}$ has length $1$), $\psi$ maps Rumer diagrams onto Rumer
diagrams (the construction in (b) does not produce intersecting
edges). Now, let $G\in R(m_1,\ldots,m_{n-1},\mu_n)$, where
$2\bigtriangleup(m_n,m_{n+1},\mu_n)$. The intersection
$\psi^{-1}(G)\cap R(m_1,\ldots,m_n,m_{n+1})$ has exactly one member
which can be constructed via the distribution process from (B) as
follows: Order the other endpoints of the $\mu_n$ valence bonds
through $n$ clockwise and attach the first $m_{n+1}-r$ of them to
$n+1$. Thus, the restriction
\[
\psi\colon R(m_1,\ldots,m_n,m_{n+1})\to
\cup_{2\bigtriangleup\left(m_n,m_{n+1},
\mu_n\right)}R(m_1,\ldots,m_{n-1},\mu_n)
\]
is a bijection, hence
\[
\rho(m_1,\ldots,m_n,m_{n+1})=
\sum_{2\bigtriangleup\left(m_n,m_{n+1},
\mu_n\right)}\rho(m_1,\ldots,m_{n-1},\mu_n).
\]
Moreover, we have
\[
\rho(m_1)=\left\{\begin{array}{ccc}
1 & \hbox{\rm\ if\ } & m_1=0  \\
0 & \hbox{\rm\ if\ } & m_1\in\mathbb{N}.  \\
\end{array}\right.
\]

{\rm (ii)} This follows from part {\rm (i)}.

\end{proof}

\subsection{The number of Rumer digrams}

\label{10.5}

Let $\rho(n,m)$ be the number of Rumer diagrams on $n$ atoms with
$m$ valence bonds.

\begin{theorem}\label{10.5.1} If $m\in\mathbb{N}$, then one has
\[
\rho(n,m)=\left|\begin{array}{ccc}
\binom{m+n-1}{n-1} & \binom{m+n-2}{n-1}  \\
\binom{m+n-1}{n-2} & \binom{m+n-2}{n-2}  \\
\end{array}\right|.
\]

\end{theorem}

\begin{proof} In accord with Theorem~\ref{10.1.10} and
Theorem~\ref{10.1.45}, we have
\[
\rho(n,m)=\sum_{m_1+\cdots+m_n=2m}\rho(m_1,\ldots,m_n)=
\sum_{m_1+\cdots+m_n=2m}N(m_1,\ldots,m_n)=
\]
\[
\sum_{m_1+\cdots+m_n=2m}\dim \inv_{m_1,\ldots,m_n}^{\left(n\right)}=
\dim\inv_m^{\left(n\right)},
\]
therefore
\begin{equation}
\rho(n,m)=\dim\inv_m^{\left(n\right)}\label{10.5.5}
\end{equation}
and Theorem~\ref{5.10.10} finishes the proof.

\end{proof}

\begin{corollary} \label{10.5.10} The monomials that correspond to
the Rumer diagrams from $R(m)$ form a basis for the
$\mathbb{C}$-linear space $\inv_m^{\left(n\right)}$ for any
$m\in\mathbb{N}_+$.

\end{corollary}

\begin{proof} Because of Theorem~\ref{10.1.1}, the
monomials~\eqref{5.5.20} that correspond to the Rumer diagrams from
$R(m)$ span the $\mathbb{C}$-linear space $\inv_m^{\left(n\right)}$.
On the other hand,~\eqref{10.5.5} yields that their number is equal
to the dimension of $\inv_m^{\left(n\right)}$ for $m\in\mathbb{N}$.
Since $\inv_0^{\left(n\right)}=\mathbb{C}$ and $\rho(n,0)=1$, we are
done.

\end{proof}

\end{document}